\documentclass[a4paper,12pt]{article}

\usepackage[ngerman, american]{babel}
\usepackage[utf8]{inputenc}
\usepackage[a4paper, left=2.5cm, right=2.5cm, top=2.5cm, bottom=2.5cm]{geometry}
\usepackage{graphicx}
\usepackage{float}
\usepackage{color,psfrag}
\usepackage{amssymb}
\usepackage{amsmath}
\usepackage{amsthm}
\usepackage{dsfont}
\usepackage{mathrsfs}
\usepackage{color}
\usepackage[automark,headsepline]{scrpage2}
\usepackage{verbatim}
\usepackage[usenames, dvipsnames]{xcolor}
\usepackage{hyperref}

\renewcommand{\phi}{\varphi}
\renewcommand{\rho}{\varrho}
\renewcommand{\epsilon}{\varepsilon}

\newcommand{\cA}{\mathcal{A}}
\newcommand{\cB}{\mathcal{B}}

\newcommand{\cE}{\mathcal{E}}

\newcommand{\cL}{\mathcal{L}}

\newcommand{\CC}{\mathbb{C}}
\newcommand{\RR}{\mathbb{R}}
\newcommand{\NN}{\mathbb{N}}

\newcommand{\dx}{{\mathrm d}}

\newcommand{\I}{\mathds{1}}

\newcommand{\ew}{\mathds{E}}
\newcommand{\supp}{\mathrm{supp}}

\newcommand{\inttail}[2]{\overline{\overline{\Pi}}_{#1}^{(#2)}}
\newcommand{\law}{\cL}

\newcommand{\bR}{\mathbb{R}}
\newcommand{\bN}{\mathbb{N}}

\newcommand{\re}{\mathrm{e}}
\newcommand{\di}{\mathrm{d}}
\newcommand{\ri}{\mathrm{i}}
\newcommand{\one}{\mathds{1}}
\newcommand{\EE}{\mathbb{E}}
\newcommand{\PP}{\mathbb{P}}
\newcommand{\LL}{\mathbb{L}}

\definecolor{violet}{rgb}{0.5, 0.0, 1.0}

\newtheorem{theorem}{Theorem}[section]
\newtheorem{lemma}[theorem]{Lemma}
\newtheorem{corollary}[theorem]{Corollary}

\newtheorem{remark}[theorem]{Remark}
\newtheorem{proposition}[theorem]{Proposition}
\newtheorem{example}[theorem]{Example}

\numberwithin{equation}{section}

\begin{document}

\author{Anita Behme\thanks{Technische Universit\"at Dresden,
		Institut f\"ur Mathematische Stochastik, 01062 Dresden, Germany, e-mail: anita.behme@tu-dresden.de}\,\,, Alexander Lindner
 and Jana Reker\thanks{Ulm University, Institute of Mathematical Finance, 89081 Ulm, Germany, e-mails: alexander.lindner@uni-ulm.de, jana.reker@uni-ulm.de}}
\title{On the law of killed exponential functionals}
\maketitle

\begin{abstract}
For two independent L\'{e}vy processes $\xi$ and $\eta$ and an exponentially distributed random variable $\tau$ with parameter $q>0$ that is independent of $\xi$ and $\eta$, the killed exponential functional is given by $V_{q,\xi,\eta} := \int_0^\tau \re^{-\xi_{s-}} \, \di \eta_s$. With the killed exponential functional arising as the stationary distribution of a Markov process, we calculate the infinitesimal generator of the process and use it to derive different distributional equations describing the law of $V_{q,\xi,\eta}$, as well as functional equations for its Lebesgue density in the absolutely continuous case. Various special cases and examples are considered, yielding more explicit information on the law of the killed exponential functional and illustrating the applications of the equations obtained. Interpreting the case~$q=0$ as~$\tau=\infty$ leads to the classical exponential functional $\int_0^\infty \re^{-\xi_{s-}} \, \di \eta_s$, allowing to extend many previous results to include killing.
\end{abstract}


\section{Introduction}
For two independent real-valued L\'evy processes $\xi$ and $\eta$, the \emph{generalised Ornstein--Uhlenbeck process $(X_t)_{t\geq 0}$ driven by $\xi$ and $\eta$} is defined by
\begin{align} \label{eq-GOU}
X_t=\re^{-\xi_t}\Bigl(\int_0^t \re^{\xi_{s-}}\dx\eta_s+X_0\Bigr),\ t\geq0,
\end{align}
where $X_0$ is a starting random variable, independent of $\xi$ and $\eta$. The generalised Ornstein--Uhlenbeck process can equivalently be  defined as the unique solution of the stochastic differential equation
$$\di X_t = X_{t-} \, \di U_t + \di \eta_t, \quad t\geq 0,$$
with starting value $X_0$, where $U$ is another L\'evy process, independent of $\eta$, and defined by the property that
\begin{equation}\label{connection-xi-U}
\cE(U)_t = \re^{-\xi_t},
\end{equation}
where $\cE(U)$ denotes the Dol\'eans--Dade stochastic exponential of $U$, cf. \cite[p.428]{MallerMuellerSzimayer2009}. Note that~\eqref{connection-xi-U} implicitly assumes $U$ not to have jumps of size less or equal to $-1$ due to the exponential function on the right-hand side being strictly positive. The generalised Ornstein--Uhlenbeck process is a Markov process, and it is known (cf. \cite[Thm.~2.1]{LindnerMaller2005}) that, provided $\xi$ and $\eta$ are not both deterministic,  it has an invariant probability distribution if and only if the stochastic integral $\int_0^t \re^{-\xi_{s-}} \, \di \eta_s$  converges almost surely to a finite limit as $t\to\infty$ (see e.g.~\cite{EricksonMaller2005} for necessary and sufficient conditions), in which case the limit random variable
$$V_{0,\xi,\eta} := \int_0^\infty \re^{-\xi_{s-}} \, \di \eta_s := \lim_{t\to \infty} \int_0^t \re^{-\xi_{s-}} \, \di \eta_s$$
is called the \emph{exponential functional of $(\xi,\eta)$}. Due to this connection, the law of $V_{0,\xi,\eta}$ is well-studied in the literature see e.g.~\cite{CarmonaPetitYor1997},~\cite{deHaanKarandikar1989}, the survey by Bertoin and Yor~\cite{BertoinYor2005}, or~\cite{BehmeLindner2015},~\cite{BehmeLindnerMaejima2016},~\cite{BertoinLindnerMaller2008},~\cite{KuznetsovPardoSavov2012},~\cite{PardoRiverovanSchaik} for some more recent results.

\medskip
Introducing jumps of size $-1$ to the process $U$ by adding an atom with mass $q>0$ to the L\'{e}vy measure of $U$, or equivalently considering $\widetilde{U}=U-N$, where $N$ denotes an independent Poisson process with parameter $q>0$, Equation~\eqref{connection-xi-U} yields a L\'{e}vy process~$\smash{\widetilde{\xi}}$ that is killed upon the first jump of $N$, i.e. after an exponential time. The stochastic differential equation
\begin{equation} \label{eq-diff-killing}
\di X_t = X_{t-} \, \di \widetilde{U}_t + \di\eta_t, \quad t \geq 0,
\end{equation}
also has a solution that is unique in law and a Markov process (see~\cite{Behme2015} and~\cite[Sect.~3]{BLRRPart1}). However, the stationary distribution of the process is now given by the {\it killed exponential functional of $(\xi,\eta)$ with parameter~$q$}
\begin{align}\label{def-killedfunctional}
V_{q,\xi,\eta} := \int_0^\tau \re^{-\xi_{s-}} \, d\eta_s,
\end{align}
(cf.~\cite[Sect.~3]{BLRRPart1}) where $\xi$ and $\eta$ are the two independent L\'evy processes defined above, and $\tau$ denotes an exponentially distributed random variable with parameter $q>0$ that is independent of~$\xi$ and~$\eta$. Interpreting the case where the parameter $q$ is equal to zero as $\tau=\infty$, we recover the exponential functional $V_{0,\xi,\eta}=\smash{\int_0^{\infty}\re^{-\xi_{s-}}\dx\eta_s}$ such that $V_{q,\xi,\eta}$ can be seen as a natural generalization of the classical case. Unless the killing is explicitly specified, we always use the term~"exponential functional" to refer to the improper integral. However, we may emphasize the difference by writing "exponential functional without killing" for~$V_{0,\xi,\eta}$.

\medskip
Killed exponential functionals have been studied in~\cite{PardoRiverovanSchaik} and~\cite{PatieSavov} among others, however, most results only cover the case $\eta_t=t$. In a recent paper \cite{BLRRPart1}, we have characterized the support of $V_{q,\xi,\eta}$ and established continuity properties of the law of $V_{q,\xi,\eta}$ for general L\'{e}vy processes $\xi$ and $\eta$. With this approach, various sufficient conditions for absolute continuity were obtained, also yielding new results for the exponential functional without killing. In this paper, we aim to derive different distributional equations for the law of the killed exponential functional, as well as functional equations for its density, and thus study the law of $V_{q,\xi,\eta}$ directly. Although it is rarely possible to give the distribution explicitly, one can do so in the following special cases.
\begin{example}\rm \cite[Thm. 2]{Yor92} \label{ex:Yor}
	Let $q>0$, $\eta_t=t$ and $\xi_t=2B_t+bt$, $t\geq 0$, for some standard Brownian motion $(B_t)_{t\geq 0}$ and $b\in \RR$, then
	$$V_{q,\xi,\eta}\overset{d}= \frac{B_{1,\beta}}{2G_{\alpha}},$$
	where $B_{1,\beta}\sim\text{Beta}(1,\beta)$ and $G_{\alpha}\sim \Gamma(\alpha,1)$ are independent, and
	$$\alpha=\frac{\gamma + b}{2},\beta=\frac{\gamma-b}{2}, \gamma=\sqrt{2q+b^2}.$$
	\end{example}

\begin{example}\rm \cite[Sect.~2]{moehle} \label{ex:ML}
		Fix $\alpha\in(0,1)$, set $q=\Gamma(1-\alpha)^{-1}$, $\eta_t=t$ and let $\xi_t$, $t\geq 0$, be a drift-free subordinator with L\'evy measure
		$$\nu_\xi(\dx x)=\frac{1}{\Gamma(1-\alpha)}\frac{e^{-x/\alpha}}{(1-e^{-x/\alpha})^{\alpha+1}}\I_{(0,\infty)}(x)\dx x.$$
Then $V_{q,\xi,\eta}$ has a Mittag-Leffler distribution with parameter $\alpha$, i.e.
its Laplace transform is a Mittag-Leffler function
$$\EE[e^{-tV_{q,\xi,\eta}}] = E_\alpha(-t)= \sum_{k\geq 0} \frac{(-t)^k}{\Gamma(1+\alpha k)},$$
and the distribution of $V_{q,\xi,\eta}$ has a Lebesgue density $f_{ML}$ given by
$$f_{ML}(s)= \frac{1}{\pi \alpha} \sum_{k\geq 0} \frac{(-1)^{k+1}}{k!} \Gamma(\alpha k + 1) s^{k-1} \sin(\pi \alpha k), \quad s>0.$$
\end{example}

In~\cite{Behme2015}, different distributional equations were derived through methods such as Laplace inversion to study both the law and the density of the exponential functional in the case where $q=0$ and $\eta$ is a subordinator. In~\cite{PardoRiverovanSchaik}, similar equations were derived in the case where $q>0$, $\eta_t=t$ and $\xi$ is a subordinator, also establishing properties of the density such as the limiting behavior at $t=0$. The case $\eta_t=t$, in particular the density of a (possibly killed) exponential functional, has also been studied in~\cite{PatieSavov}. Another approach to derive a distributional equation in the case $q=0$ has been developed in~\cite{KuznetsovPardoSavov2012}. Here, the main tools used are a moment condition and results from Schwartz theory of distributions, with the latter being applicable since the domain of the infinitesimal generator of the underlying Markov process includes the test functions. While studying the method, however, we found that there has been an oversight in the proof of~\cite[Thm.~2.3]{KuznetsovPardoSavov2012} such that the result is not applicable in all of the mentioned cases (see Remark~\ref{validity-analysis}).

After establishing some preliminaries in Section~\ref{s2}, we recall the connection between the killed exponential functional and the solution of~\eqref{eq-diff-killing} from~\cite[Thm.~3.1]{BLRRPart1} in Section~\ref{s3}. From this, we calculate the infinitesimal generator of the process, which is the starting point of the analysis. An integro-differential equation for the characteristic function of the killed exponential functional is  then derived in Section~\ref{s4} using the methods from~\cite{Behme2015} and~\cite{BehmeLindner2015}. Section~\ref{s5} is concerned with deriving a general equation for the law of $V_{q,\xi,\eta}$ through Schwartz theory of distributions using a similar approach to~\cite{KuznetsovPardoSavov2012}, also discussing applications and special cases such as the case without killing. The proofs for the results in this section are given in Section~\ref{s6}. In the final section, we collect several more examples to derive explicit results for the law of the killed exponential functional, also showing that the previously derived equations can be solved explicitly in special cases.

%%%%%%%%%%%%%%%%%%%%%%%%%%%%%%%%%%%%%%%%%%
\section{Preliminaries}\label{s2}
%%%%%%%%%%%%%%%%%%%%%%%%%%%%%%%%%%%%%%%%%%
A real-valued L\'evy process $L=(L_t)_{t\geq0}$ is a stochastic process having stationary and independent increments that starts in 0 and has almost surely c\`adl\`ag paths, i.e. paths that are right-continuous with finite left-limits. By the L\'evy-Khintchine formula (see e.g.~\cite[Thm. 8.1]{Sato2013}), its characteristic function is given by
$$\phi_{L_t}(z)=\ew \re^{\ri z L_t} = \exp (t \psi_L(z)) , \quad z\in \bR,$$
where $\psi_L$ denotes the characteristic exponent satisfying
\begin{align*}
\psi_L(z)=\ri \gamma_Lz-\sigma_L^2z^2/2 +\int_{\RR}(\re^{\ri zx}-1-\ri zx\one_{\{|x|\leq1\}})\nu_L(\dx x), \quad z\in \bR.
\end{align*}
Here, $\sigma_L^2\geq 0$ is the \emph{Gaussian variance}, $\nu_L$ is the \emph{L\'evy measure} and $\gamma_L$ the \emph{location parameter} of $L$. The characteristic triplet of $L$ is denoted by $(\sigma_L^2,\nu_L,\gamma_L)$. Whenever the L\'{e}vy measure satisfies the condition $\int_{|x|\leq 1} |x|\nu_L(\di x)<\infty$, we can also use the L\'evy-Khintchine formula in the form
\begin{align*}
	\psi_L(u)&= \ri \gamma^0_L u -  \sigma_L^2 u^2/2  +
	\int_{\RR} (\re^{\ri  u x} -1 ) \nu_L(\di x), \quad u\in \bR,
\end{align*}
and call $\gamma_L^0$ the {\it drift} of $L$. See e.g. \cite{Sato2013} for any further information on L\'evy processes.

\medskip
For any c\`{a}dl\`{a}g process $X$, we denote by $X_{s-}$ the left-hand limit of $X$ at time~${s\in(0,\infty)}$ and by $\Delta X_s=X_s-X_{s-}$ its jumps. The law of a random variable~$Y$ is denoted by~$\cL(Y)$. Further, we write $\smash{\overset{d}{=}}$ for equality in distribution. If the random variable $Y$ is exponentially distributed with parameter $q\geq0$, we set~$Y\smash{\overset{d}{=}}{\rm{Exp}}(q)$ with the case $q=0$ being interpreted as $Y\equiv\infty$ almost surely. The (one-dimensional) Lebesgue measure is denoted by $\lambda$ and absolute continuity, as well as densities, are always assumed to be with respect to $\lambda$ unless stated otherwise. The Dirac measure at $x\in\RR$ is denoted by $\delta_x$ and $\I_A$ denotes the indicator function of the set $A\subset\RR$. We further write ''a.s.'' and ''a.e.'' to abbreviate ''almost surely'' and ''almost every(where)'', respectively. The image measure of a general measure~$m$ under~a mapping~$g$ is denoted by $g(m)$. When given two measures~$m_1,m_2$, their convolution is denoted by $m_1\ast m_2$. We use the same notation for the convolution of a measure $m$ with an integrable function $g$, by which we mean the function with value~$(m\ast g)(x)=\smash{\int_{\RR}g(x-y)m(\dx y)}$ at $x\in\RR$.

\medskip
When considering integrals, we assume the integral bounds to be included when using the notation $\smash{\int_a^b}$ for $a,b\in\RR$ with $a\leq b$ and indicate that the left or right bound is excluded by writing~$\smash{\int_{a+}^b}$ or $\smash{\int_a^{b-}}$, respectively. The notation $\smash{\int_{a+}^b}$ for $a\geq b$ is to be interpreted as~$-\smash{\int_{b+}^a}$, such that mappings of the form $x\mapsto \int_{0+}^xh(s)m(\dx s)$ are càdlàg functions. Integrals like~$\smash{\int_0^t \re^{-\xi_{s-}} \, \di \eta_s}$ are interpreted as integrals with respect to semimartingales as e.g. in \cite{Protter2005}. Given a L\'evy process $L=(L_t)_{t\geq 0}$ (or more generally, a semimartingale), its stochastic exponential $\cE(L)= (\cE(L)_t)_{t\geq 0}$ is the unique semimartingale $Z=(Z_t)_{t\geq 0}$ that satisfies the stochastic differential equation $\di Z_t = Z_{t-}\, \di L_t$ with~$\cE(L)_0=1$, i.e. which satisfies
$Z_t = 1 + \smash{\int_0^t Z_{t-}}  \, \di L_t$ for all $t\geq 0.$ By the Doleans-Dade formula (\cite[Thm.~II.37]{Protter2005}), it is given by
\begin{equation} \label{eq-StoExp}
\cE (L)_t = \re^{L_t - t \sigma_L^2/2} \prod_{0<s\leq t} (1 + \Delta L_s) \re^{-\Delta L_s},\quad t\geq0.
\end{equation}
It follows that $\cE(L)$ is almost surely strictly positive for all times if and only if ${(1+\Delta L_s) > }0$ for all times, i.e. if $\nu_L((-\infty,-1]) = 0$. Now let $\xi$ be a L\'evy process on $\bR$. Then it is easy to see that $U=(U_t)_{t\geq 0}$, defined by
\begin{equation} \label{eq-U}
U_t = -\xi_t + t \sigma_\xi^2 / 2 + \sum_{0<s\leq t} \left( \re^{-\Delta \xi_s} - 1 + \Delta \xi_s\right)
\end{equation}
is also a L\'evy process, and using \eqref{eq-StoExp} it is readily checked that $\cE(U)_t = \re^{-\xi_t}$. Taking the logarithm in \eqref{eq-StoExp}  recovers $\xi$ from $U$ via
\begin{equation} \label{eq-U2}
\xi_t = -U_t + t \sigma_U^2/2 - \sum_{0<s\leq t} \big( \ln (1+\Delta U_s) - \Delta U_s\big), \quad t\geq 0.
 \end{equation}
It follows that \eqref{eq-U} defines a bijection from the class of all L\'evy processes $\xi$ to the class of L\'evy processes $U$ with $\nu_U ((-\infty,-1]) = 0$, with inverse given by \eqref{eq-U2}, and the relation is described by $\cE(U) = \re^{-\xi}$. The characteristic triplet of $\xi$ in terms of that of $U$ has been derived in \cite[Lemma 3.4]{BehmeLindnerMaller2011} and is given by
\begin{align*}
\sigma_\xi^2 &= \sigma_U^2, \quad \nu_\xi = g(\nu_U), \\
 \gamma_\xi &= - \gamma_U + \sigma_U^2/2 + \int_{(-1,\infty)} \left( x \one_{\{|x|\leq 1\}} - (\ln(1+x))\one_{\{x \in [\re^{-1}-1,\re -1 ]\}} \right)\, \nu_U(\di x),
 \end{align*}
where $g:(-1,\infty) \to \bR$ is defined by $g(x) = -\ln(1+x)$. From this it is readily seen that the characteristic triplet of $U$ in terms of $\xi$ is expressed by
\begin{equation*} \label{eq-rel2}
\sigma_U^2 = \sigma_\xi^2, \quad \nu_U = h(\nu_\xi), \quad \gamma_U = - \gamma_\xi - \sigma_\xi^2/2 + \int_{[-\log 2, \infty)} \left[ \left( \re^{-x} - 1 \right) + x \one_{\{ |x|\leq 1\}} \right] \, \nu_\xi (\di x),
\end{equation*}
where $h:\bR \to (-1,\infty)$ is given by $h(x) = \re^{-x} - 1$.
We also see from \eqref{eq-U} and \eqref{eq-U2} that~$U$ is of finite variation if and only if $\xi$ is, in which case the drifts are related by $\smash{\gamma_U^0 = - \gamma_\xi^0}$.

\medskip
Throughout the analysis, we denote the space of continuous functions $\RR\rightarrow\RR$ by~$C(\RR)$. The subspaces of bounded functions, functions vanishing at infinity and compactly supported functions are referred to as $C_b(\RR)$, $C_0(\RR)$ and $C_c(\RR)$, respectively. For a number~$n\in\NN$, we denote by $C^n(\RR)$ the space of functions $\RR\rightarrow\RR$ that are $n$ times continuously differentiable with $C^{\infty}(\RR)$ denoting that the property holds for every $n$. Functions in the subspaces $C^n_0(\RR)$ are $n$ times continuously differentiable with the function itself, as well as the first $n$ derivatives vanishing at infinity. The spaces $C_c^n(\RR)$ are defined analogously and~$C_c^{\infty}(\RR)$ is also referred to as the space of test functions. The domain of a linear operator~$\cA$ is denoted by $D(\cA)$.

\section{Killed exponential functionals as invariant distributions of Markov processes}\label{s3}
%%%%%%%%%%%%%%%%%%%%%%%%%%%%%%%%%%%%%%%%%%
It was shown in~\cite[Thm.~3.1]{BLRRPart1} that the law of the killed exponential functional describes the stationary distribution of a Markov process. In this section, we aim to build on this result by explicitly calculating the infinitesimal generator of this Markov process. First, we recall the central result of~\cite[Thm.~3.1]{BLRRPart1} for the reader's convenience. Note that the case $q=0$ was already shown in~\cite[Thm.~2.1]{BehmeLindnerMaller2011}.

\begin{proposition}\label{prop-sdewithkilling}
	Let $\xi$ and $\eta$ be two independent L\'evy processes and $q\in [0,\infty)$. Define the L\'evy process $U$ by \eqref{eq-U} such that
		$\mathcal{E}(U) =\re^{-\xi}$
		and let $N$ either be a Poisson process with parameter $q>0$ that is independent of $U$ and $\eta$ or $N\equiv0$ if $q=0$. Define
		$$\widetilde{U} := U - N .$$
		Then $\widetilde{U}$ is a L\'evy process with characteristic triplet $(\sigma_U^2,\nu_U+q\delta_{-1},\gamma_U-q)$. Further,
\begin{equation}\label{eq-representations}
V_{q,\xi,\eta}\stackrel{d}{=}\int_0^\infty \mathcal{E}(\widetilde{U})_{s-} \, \di \eta_s
\end{equation}
whenever the right-hand side converges almost surely. In this case
$\smash{\law(V_{q,\xi,\eta})}$ is the unique invariant probability measure of the Markov process $\smash{{ \widetilde{V}=(\widetilde{V}_t)_{t\geq 0}}}$ satisfying the stochastic differential equation
	\begin{equation} \label{eq-diff1}
	\di \widetilde{V}_t = \widetilde{V}_{t-} \, \di \widetilde{U}_t + \di\eta_t, \quad t \geq 0,
	\end{equation}
	with starting random variable $\smash{\widetilde{V}_0}$ independent of $\widetilde{U}$ and $\eta$.
\end{proposition}

The infinitesimal generator $\cA^{\widetilde{V}}$ of $(\widetilde{V}_t)_{t\geq 0}$ is the linear operator defined by
\begin{displaymath}
\cA^{\widetilde{V}} f (x)= \lim_{t\to 0} \frac{\EE^x [f(\widetilde{V}_t^x)] - f(x)}{t}, \quad x\in \bR,
\end{displaymath}
on the set of functions $f\in C_b(\bR)$ for which this limit exists uniformly in $x$. Here, $\widetilde{V}_t^x$ denotes the solution of~\eqref{eq-diff1} with initial value $\widetilde{V}_0^x=x$ and $\ew^x$ denotes the corresponding expectation. As a starting point of the analysis, consider the generalized Ornstein--Uhlenbeck process, i.e. the Markov process given by
$$V_t^x = x + \int_0^t V_{s-}^x \, \dx U_t + \eta_t,$$
which is the solution of the differential equation $dV_t = V_{t-} \, dU_t + d\eta_t$ with starting random variable $V_0^x = x$. As shown in \cite[Thm.~3.1, Cor.~3.2,  Cor.~3.3]{BehmeLindner2015}, $(V_t^x)_{t\geq 0}$ is a (rich) Feller process and the domain of its infinitesimal generator $\cA^{V}$ contains the space
$$C_{0,pl}^2(\bR) := \left\{ f \in C^2_0(\bR) : \lim_{|x|\to\infty} \left( |xf'(x)| + |x^2 f''(x)| \right) = 0\right\}, $$
where the added subscript $pl$ refers to the power law decay of the derivatives, on which~$\cA^V$ acts by
\begin{align}
\cA^{V}f(x) &=\cA^\eta f(x) - f'(x) x \gamma_\xi + \frac12 (f''(x) x^2 + f'(x) x) \sigma_\xi^2 \nonumber \\
&\quad+ \int_\bR \left( f(x\re^{-y}) - f(x) + f'(x) xy \I_{\{|y|\leq 1\}} \right) \nu_\xi(\dx y)\nonumber\\
&=\cA^\eta f(x)+xf'(x)\gamma_U+\frac{1}{2}x^2f''(x)\sigma_U^2\nonumber\\
&\quad+\int_{\RR}\big(f(x+xy)-f(x)-xyf'(x)\I_{\{|y|\leq1\}}\big)\nu_{U}(\dx y), \label{eq-gen1}
\end{align}
where $\cA^{\eta}$ denotes the infinitesimal generator of the L\'{e}vy process $\eta$ given by
\begin{align*}
\cA^\eta f(x)=\gamma_\eta f'(x) + \frac12 \sigma_\eta^2f''(x)+\int_\bR \left( f(x+y) - f(x) - f'(x)y \mathbf{1}_{|y|\leq 1} \right) \nu_\eta(\dx y)
\end{align*}
for ${f\in C^2_{0,pl}(\RR)}$. From this we can derive the generator of $\widetilde{V}$ as follows.

\begin{theorem} \label{thm-gen1}
	Let $q\in[0,\infty)$, $\smash{(\widetilde{V}_t)_{t\geq 0}}$ as defined in \eqref{eq-diff1} and assume that $V_{0,\xi,\eta}$ converges almost surely whenever $q=0$ is considered. Then the set $C^2_{0,pl}(\bR)$ is contained in $D(\cA^{\widetilde{V}})$, and for~$f\in C^2_{0,pl}(\bR)$ we have
	$$\cA^{\widetilde{V}} f(x) = \cA^{V} f(x) + q \left( f(0) - f(x) \right),$$
	with $\cA^{V}$ as given in \eqref{eq-gen1}.
\end{theorem}
\begin{proof} The case $q=0$ was shown in~\cite{BehmeLindner2015}. Let $q>0$ and $f\in C_{0,pl}^2(\bR)$. Then for each $t>0$,
	\begin{eqnarray}
	\frac{\EE^x [f(\widetilde{V}_t^x)] - f(x)}{t} & = & \frac{\EE^x [f(\widetilde{V}_t^x|N_t=0)] \PP(N_t=0) - f(x)}{t} \nonumber \\
	& & + \EE^x [f(\widetilde{V}_t^x)|N_t=1] \frac{\PP(N_t=1)}{t} \nonumber \\
	& & + \EE^x [f(\widetilde{V}_t^x)|N_t\geq 2] \frac{\PP(N_t\geq 2)}{t}. \label{eq-gen2}
	\end{eqnarray}
	Since $f$ is bounded and $\PP(N_t\geq 2) = o(t)$ as $t\to 0$, the last term tends to 0, uniformly in $x\in \bR$, as $t\to 0$. Denote the time of the last jump of $N$ before $t$ by $T(t)$. Then
	$$\widetilde{V}_t = \re^{-\xi_t} \left( x + \int_0^t \re^{\xi_{s-}}\, \dx\eta_s \right) \mathds{1}_{\{N(t)=0\}} + \left( \re^{-(\xi_t - \xi_{T(t)})} \int_{(T(t),t]} \re^{\xi_s-\xi_{T(t)}} \, \dx\eta_s\right) \mathds{1}_{\{N(t)\geq 1\}}$$ by \cite[Prop. 3.2]{BehmeLindnerMaller2011}. Since $\PP(N_t=1) = qt \re^{-qt}$, we conclude from this that
	$$\lim_{t\to 0} \EE^x [f(\widetilde{V}_t^x)|N_t=1] \frac{\PP(N_t=1)}{t} = f(0) q,$$ uniformly in $x$. Finally, since $\PP(N_t=0) = e^{-qt}$ we can write
	\begin{align*}
		\frac{\EE^x[f(\widetilde{V}_t^x|N_t=0)] P(N_t=0) - f(x)}{t} = \frac{\EE^x[f(V_t^x)] - f(x)}{t} + \frac{\re^{-qt}-1}{t} \EE^x[f(V_t^x)].
	\end{align*}
	The first of these terms converges uniformly in $x$ to $\cA^{V}f(x)$ as $t\to 0$, and the second uniformly to $-q f(x)$ since $\EE^x[f(V_t^x)]$ converges uniformly to $f(x)$ since $(V_t^x)_{t\geq 0}$ is a Feller process (\cite[Thm. 3.1]{BehmeLindner2015}).
	Together with \eqref{eq-gen2} this gives the claim.
\end{proof}

\begin{remark}\label{remark-alternative-forms}\rm (i) Alternatively, the above theorem could be shown following the proof of Theorem 3.1 and Corollary 3.2 in \cite{BehmeLindner2015} and replacing $U$ by $\widetilde{U}$ to allow for jumps of size~$-1$. Observing that the characteristics $\gamma_U$ and $\gamma_{\widetilde{U}}$ differ by~${q = \smash{- \int_{\{-1\}} y \nu_{\widetilde{U}}(dy)}}$ then leads to the same result.\\
(ii) Aside from the expression in Theorem~\ref{thm-gen1}, the operator $A^{\widetilde{V}}$ can also be given in terms of the characteristics of $\smash{\widetilde{U}}$. As $\nu_{\widetilde{U}}(\{-1\})=q$ we have for$f\in C_{0,pl}^2(\bR)$ that
\begin{align}
\cA^{\widetilde{V}} f(x)&=\cA^\eta f(x)+xf'(x)\gamma_{\widetilde{U}}+\frac{1}{2}x^2f''(x)\sigma_{\widetilde{U}}^2\nonumber\\
&\quad+\int_{\RR}\big(f(x+xy)-f(x)-xyf'(x)\I_{[-1,1]}(y)\big)\nu_{\widetilde{U}}(\dx y)\label{eq-genkilledgou}
\end{align}
which is~\eqref{eq-gen1} with $U$ replaced by $\widetilde{U}$.
\end{remark}

The key to deriving the equations describing $\cL(V_{q,\xi,\eta})$ in the following sections lies in the fact that the law of the killed exponential functional is the unique invariant probability law of the Markov process in \eqref{eq-diff1} and thus the equation
\begin{equation}\label{eq-invmeasure}
\int_{\bR} \cA^{\widetilde{V}}f(x) \, \law(V_{q,\xi,\eta})(\dx x) = 0
\end{equation}
holds for every function $f$ in the domain of the operator $\cA^{\widetilde{V}}$ (see e.g.~\cite[Thm.~3.37]{Liggett2010}; although the proof is given for Feller processes only, one can see from the argument given that this must hold true also for invariant probability measures of general Markov processes). In view of Theorem~\ref{thm-gen1}, this is in particular satisfied for $f\in C^2_{0,pl}(\RR)$. We also note the following special case as a key tool for Section~\ref{s5}.

\begin{corollary}\label{cor-testfunctions}
The space $C_c^{\infty}(\RR)$ is a subset of $\smash{D(\cA^{\widetilde{V}})}$ and~\eqref{eq-invmeasure} holds for every test function $f$.
\end{corollary}

\section{Distributional Equations Derived by Fourier and Laplace Methods}\label{s4}
In this section, we use the infinitesimal generator obtained in Theorem~\ref{thm-gen1} to derive distributional equations for the law of the killed exponential functional, as well as a functional equation to describe its density, using the method developed in~\cite{BehmeLindner2015} for the case without killing. Throughout the analysis, we set $\cL(V_{q,\xi,\eta})=\mu$ and denote its characteristic function by $\phi_{q,\xi,\eta}$. The following conclusion now follows in complete analogy to Theorem~4.1 and Corollary~4.3 in \cite{BehmeLindner2015}, using Lemma~4.2 of \cite{BehmeLindner2015}. For convenience, the following corollary is given in the characteristics of the original L\'{e}vy process $\xi$, as well as in the characteristics of $\smash{\widetilde{U}}$ with $\nu_{\widetilde{U}}(\{-1\})=q$.

\begin{corollary} \label{cor-cf1}
	Let $q\geq0$ and assume that the exponential functional converges a.s. whenever $q=0$ is considered. Further, let $h\in C_c^\infty(\bR)$ such that $h(x)=1$ for~${|x|\leq 1}$ and~$h(x)=0$ for~$|x|\geq 2$. Set $h_n(x) := h(x/n)$ and ${f_{u,n}(x) = \re^{iux} h_n(x)}$ for $u\in \bR$, $n\in \bN$, and $x\in \bR$. Then
		\begin{align}
		\psi_\eta(u) \phi_{V_{q,\xi,\eta}}(u)
		& =  q (\phi_{V_{q,\xi,\eta}}(u) -1)\nonumber \\
		& \quad + \lim_{n\to\infty} \left( \gamma_\xi \int_{\bR} x f_{u,n}'(x) \, \mu(\dx x)
		- \frac{\sigma_\xi^2}{2} \int_{\bR} (x^2 f_{u,n}''(x) + x f_{u,n}'(x)) \, \mu(\dx x) \right.\nonumber \\
		&  \qquad - \left. \int_{\bR} \int_{\bR} (f_{u,n}(x \re^{-y}) - f_{u,n}(x) + xy f_{u,n}'(x) \I_{\{|y|\leq 1\}}) \, \nu_\xi(\dx y) \, \mu(\dx x) \right)\nonumber \\
			&= - \lim_{n\to\infty} \left( \gamma_{\widetilde{U}} \int_{\bR} x f_{u,n}'(x) \, \mu(\dx x)
			+ \frac{\sigma_{\widetilde{U}}^2}{2} \int_{\bR} x^2 f_{u,n}''(x) \, \mu(\dx x) \right. \label{eq-limitchar}\\
			&  \qquad + \left. \int_{\bR} \int_{[-1,\infty)} (f_{u,n}(x+xy) - f_{u,n}(x) - xy f_{u,n}'(x) \I_{\{|y|\leq 1\}}) \, \nu_{\widetilde{U}}(\dx y) \, \mu(\dx x) \right)\nonumber
	\end{align}
	 for all $u\in\RR$. If additionally $\smash{\EE V_{q,\xi,\eta}^2 < \infty}$, then, for all $u\in\RR$,
	\begin{align}
	\psi_\eta(u) \phi_{V_{q,\xi,\eta}} (u)
	& =  q (\phi_{V_{q,\xi,\eta}}(u) -1) + \gamma_\xi u
	\phi_{V_{q,\xi,\eta}}'(u) - \frac{\sigma_\xi^2}{2} \left( u^2
	\phi_{V_{q,\xi,\eta}}''(u) + u \phi_{V_{q,\xi,\eta}}'(u)\right) \nonumber \\
	& \quad - \int_{\bR} \left( \phi_{V_{q,\xi,\eta}} (u \re^{-y}) -
	\phi_{V_{q,\xi,\eta}}(u) + u y \phi_{V_{q,\xi,\eta}}'(u) \mathbf{1}_{|y|\leq 1}
	\right) \, \nu_\xi(\dx y)\nonumber\\
	& = -\gamma_{\widetilde{U}} u
	\phi_{V_{q,\xi,\eta}}'(u) - \frac{\sigma_{\widetilde{U}}^2}{2} u^2
	\phi_{V_{q,\xi,\eta}}''(u) \nonumber \\
	& \quad - \int_{[-1,\infty)} \left( \phi_{V_{q,\xi,\eta}} (u +uy) -
	\phi_{V_{q,\xi,\eta}}(u) - u y \phi_{V_{q,\xi,\eta}}'(u) \I_{\{|y|\leq 1\}}
	\right) \, \nu_{\widetilde{U}}(\dx y) \nonumber\\
	&= -\EE\left[e^{iuV_{q,\xi\eta}} \psi_{\widetilde{U}}(uV_{q,\xi\eta})\right]. \label{eq-diffeqchar}
	\end{align}
\end{corollary}

\begin{remark}\rm
Observe that the integral with respect to $\nu_{\widetilde{U}}$ does not vanish even if $\xi$ (and hence $U$) is a Brownian motion with drift due to the added point mass at $-1$.
\end{remark}

Equation \eqref{eq-diffeqchar} can be solved in special cases, some of which are discussed in Section~\ref{s7}. Note that it has been shown in \cite[Thm. 3.1]{Behme2011}, that the precondition $\EE V_{q,\xi,\eta}^2<\infty$ is fulfilled~if
\begin{equation}\label{eq-condsecondmoment}
\EE[U_1^2]<\infty,\quad \EE[\eta_1^2]<\infty, \quad 2\EE[U_1]+\mathrm{Var}(U_1)<q,
\end{equation}
and $\lim_{t\to\infty} \cE(\smash{\widetilde{U}})_t=0$ a.s., the latter obviously being satisfied whenever $q>0$. If $\eta$ is a subordinator, an equation similar to~\eqref{eq-diffeqchar} also holds for the Laplace transforms without any moment condition. Let $\LL_Y(u)$ denote the Laplace transform of the law of a random variable $Y$, e.g. $\LL_{V_{q,\xi,\eta}} (u)=\EE[\re^{-uV_{q,\xi,\eta}}]$, $u\geq0$. Similar to Remark 4.5 in \cite{BehmeLindner2015} we obtain
\begin{align*}
 \big(\ln\LL_{\eta_1}(u)\big) \LL_{V_{q,\xi,\eta}} (u)
= & q (\LL_{V_{q,\xi,\eta}}(u)-1)  - \gamma_\xi u
\EE[V_{q,\xi,\eta}\re^{-uV_{q,\xi,\eta}}]\\ & - \frac{\sigma_\xi^2}{2} \left( u^2
\EE[V_{q,\xi,\eta}^2 \re^{-uV_{q,\xi,\eta}}] - u \EE[V_{q,\xi,\eta}\re^{-uV_{q,\xi,\eta}}] \right) \nonumber \\
& - \int_{\bR} \left( \LL_{V_{q,\xi,\eta}} (u \re^{-y}) -
\LL_{V_{q,\xi,\eta}}(u) - u y \EE[V_{q,\xi,\eta}\re^{-uV_{q,\xi,\eta}}] \mathbf{1}_{|y|\leq 1}
\right) \, \nu_{\xi}(\dx y), \nonumber
\end{align*}
for $u>0$, rearranging which yields
\begin{align}
\frac{\ln \LL_{\eta_1}(u)}{u} \LL_{V_{q,\xi,\eta}} (u)
= & q \frac{\LL_{V_{q,\xi,\eta}}(u)-1}{u}  + \Big(\gamma_\xi-  \frac{\sigma_\xi^2}{2}\Big)
\LL'_{V_{q,\xi,\eta}}(u) - \frac{\sigma_\xi^2}{2}  u\LL''_{V_{q,\xi,\eta}}(u)\nonumber \\
& - \int_{\bR} \left( \frac{\LL_{V_{q,\xi,\eta}} (u \re^{-y})}{u} -
\frac{\LL_{V_{q,\xi,\eta}}(u)}{u} +  y  \LL'_{V_{q,\xi,\eta}}(u) \mathbf{1}_{|y|\leq 1}
\right) \, \nu_{\xi}(\dx y).\label{eq-laplace}
\end{align}
Restricting the jump part of $\xi$ to be of finite variation,~\eqref{eq-laplace} reduces to
\begin{align*}
\frac{\ln \LL_{\eta_1}(u)}{u} \LL_{V_{q,\xi,\eta}} (u)
= & q \frac{\LL_{V_{q,\xi,\eta}}(u)-1}{u}  + \Big(\gamma^0_\xi-  \frac{\sigma_\xi^2}{2}\Big)
\LL'_{V_{q,\xi,\eta}}(u) -\frac{\sigma_\xi^2}{2}  u\LL''_{V_{q,\xi,\eta}}(u) \\
& - \int_{\bR} \left( \frac{\LL_{V_{q,\xi,\eta}} (u \re^{-y})}{u} -
\frac{\LL_{V_{q,\xi,\eta}}(u)}{u}
\right) \, \nu_{\xi}(\dx y)
\end{align*}
and we can derive a functional equation for the density of $V_{q,\xi,\eta}$ in the absolutely continuous case by Laplace inversion. The proof is in complete analogy to the proof for the case $q=0$ given in Theorem 2.1 in \cite{Behme2015} and hence omitted. For $q\geq0$ we obtain the following result.

\begin{proposition} \label{prop-dens-form1}
	Assume that the jump part of $\xi$ is of finite variation and $\eta$ is a subordinator, i.e. $\ln \LL_{\eta_1}(u)=-\gamma_\eta^0 u - \smash{\int_{(0,\infty)} (1-\re^{-uy})\nu_\eta (\dx y)}$ for $u\geq0$. Further
	assume that~$\cL(V_{q,\xi,\eta})=\mu$ is absolutely continuous with density $f_{\mu}$ and, whenever $\sigma_{\xi}^2\neq0$, the function~$z\mapsto z^2f_{\mu}(z)$ is absolutely continuous on $[0,z]$ for all $z>0$. Then $f_{\mu}(z)$ fulfills  for~$\lambda$-a.e. $z>0$
	\begin{align}
	&\gamma_\eta^0f_{\mu}(z)-\Big(\gamma_{\xi}^0+\frac{\sigma_{\xi}^2}{2}\Big)zf_{\mu}(z)-\frac{\sigma_{\xi}^2}{2}z^2f_{\mu}'(z)-q\int_z^{\infty} f_{\mu}(s) \dx s \label{eq-densitylaplace}\\
&\quad\quad=  \int_z^\infty \nu_\xi((\ln \tfrac{s}{z}, \infty)) f_{\mu}(s) \dx s-\int_0^z \big(\nu_\xi((-\infty, \ln \tfrac{s}{z}))+ \nu_\eta((z-s,\infty))\big) f_{\mu}(s) \dx s.\nonumber
	\end{align}
\end{proposition}

Various sufficient conditions for absolute continuity of $\mu$ are given in Theorems~6.18 and~6.14 of~\cite{BLRRPart1}. Nevertheless, there are cases where Proposition~\ref{prop-dens-form1} is not applicable, e.g. if $\eta$ is not a subordinator, if $\smash{\int_{-1}^1|x|\nu_{\xi}(\dx x)}=\infty$, or if $\mu$ is not absolutely continuous. In the next section, we derive a general equation for the law of $V_{q,\xi,\eta}$ without a priori assumptions from which Proposition~\ref{prop-dens-form1} is reobtained as a special case (see Remark~\ref{rem-literature}). The proof given in this section, however, is comparably shorter and less technical.
\begin{remark}
\rm Observe that we obtain the functional equation given in~(2.3) of~\cite{PardoRiverovanSchaik} in the special case of $\eta_t=t$ and $\xi$ being a subordinator, as $V_{q,\xi,t}$ is always absolutely continuous by~\cite{PatieSavov}.
\end{remark}

%%%%%%%%%%%%%%%%%%%%%%%%%%%%%%%%%%%%%%%%%%
\section{Distributional Equations Derived by Schwartz Theory of Distributions}\label{s5}
%%%%%%%%%%%%%%%%%%%%%%%%%%%%%%%%%%%%%%%%%%

In this section, we give distributional equations for the law of the killed exponential functional using Schwartz theory of distributions, where we follow a similar approach as used in~\cite[Thm.~2.2]{KuznetsovPardoSavov2012} for the exponential functional without killing. While studying the method, we found a small oversight in the proof of said theorem which results in the distributional equation not being applicable in all claimed cases. This is discussed in Remark~\ref{validity-analysis}. However, we also found that the method works when killing is included and that the moment condition~$\EE|\xi_1|,\EE|\eta_1|<\infty$ of~\cite{KuznetsovPardoSavov2012} is not needed to arrive at the desired conclusion in both cases. Compared to Section~\ref{s4}, we now rely more on technical auxiliary results. As a consequence, many a priori assumptions needed in the previous section can be dropped. The main result of this section is Theorem~\ref{main}, which establishes a connection between the characteristic triplets of the processes $\eta$ and $\smash{\widetilde{U}}$, and the law of the corresponding killed exponential functional~$V_{q,\xi,\eta}$. From this, we directly obtain a functional equation for the density in the absolutely continuous case as well as, similar to~\cite[Cor.~2.3]{KuznetsovPardoSavov2012}, a criterion for absolute continuity and continuity or smoothness of the density that extends the one given in Corollary~6.15 of~\cite{BLRRPart1} for the exponential functional without killing to the case~$q>0$. Further, we discuss different special cases. Recall that the process $U$ is constructed from~$\xi$ via $\re^{-\xi}=\cE(U)$ and that $\smash{\widetilde{U}}$ is obtained from adding a point mass of~${q>0}$ at $-1$ to the L\'{e}vy measure of~$U$. To alleviate some of the notation, we characterise the functions involved in Theorem~\ref{main} in the following lemma.

\begin{lemma}\label{lem-functions}
Let $\xi,\eta$ be two independent L\'{e}vy processes such that $\eta$ is not the zero process and $q\geq0$. Further, define the functions $B_{\eta},B_{\widetilde{U}},S_{\eta},S_{\widetilde{U}}$ by
\begin{align}
B_{\eta}:\RR\rightarrow\RR,\quad B_{\eta}(z)&=\begin{cases}-\nu_{\eta}(-\infty,\min\{z,-1\}),\ &\text{if}\ z<0,\\ 0,&\text{if}\ z=0,\\
\nu_{\eta}((\max\{z,1\},\infty)),\ &\text{if}\ z>0,\end{cases}\label{def-B-eta}\\
B_{\widetilde{U}}:[1,\infty)\rightarrow[0,\infty),\quad B_{\widetilde{U}}(z)&=\begin{cases}0,&\text{if}\ z=1,\\ \nu_{\widetilde{U}}((\max\{z-1,1\},\infty)),&\text{if}\ z>1,\end{cases}\label{def-B-U}\\
S_{\eta}:\RR\rightarrow[0,\infty),\quad S_{\eta}(z)&=\begin{cases}\int_{-\infty}^z(z-y)\nu_{\eta}|_{[-1,1]}(\dx y),\ &\text{if}\ z<0,\\0, &\text{if}\ z=0,\\ \int_z^{\infty}(y-z)\nu_{\eta}|_{[-1,1]}(\dx y), &\text{if}\ z>0,\\ \end{cases}\label{def-S-eta}\\
S_{\widetilde{U}}:[0,\infty)\rightarrow[0,\infty),\quad S_{\widetilde{U}}(z)&=\begin{cases} \int_{-\infty}^{z-1}(z-1-y)\nu_{\widetilde{U}}\big|_{[-1,1]}(\dx y),\ &\text{if}\ z\in[0,1),\\ 0,&\text{if}\ z=1,\\ \int_{z-1}^{\infty}(y-z+1)\nu_{\widetilde{U}}\big|_{[-1,1]}(\dx y),\ &\text{if}\ z>1.\end{cases}\label{def-S-U}
\end{align}
Then both $B_{\widetilde{U}}$ and $B_{\eta}$ are bounded and hence locally integrable with respect to $\lambda$ and both~$S_{\eta}$ and $z\mapsto S_{\widetilde{U}}(z+1)$, $z\in\RR$ are integrable with respect to $\lambda$. In particular, the convolution $B_{\eta}\ast\mu$ is defined everywhere and bounded and the convolution $S_{\eta}\ast\mu$ is defined everywhere, is $\lambda$-a.e. finite and integrable. Further, the functions $z\mapsto\int_{0+}^z(B_{\eta}\ast\mu)(x)\dx x$ and~${z\mapsto\int_{0+}^z\int_{0+}^tB_{\widetilde{U}}(\tfrac{t}{x})\mu(\dx x)\dx t}$ are locally integrable with respect to $\lambda$.
\end{lemma}

\begin{proof}
First, note that $|B_{\eta}(z)|\leq\nu_{\eta}(\RR\setminus[-1,1])<\infty$ and $|B_{\widetilde{U}}(z)|\leq\nu_{\widetilde{U}}([1,\infty))<\infty$ implies that $B_{\eta}$ and $B_{\widetilde{U}}$ are bounded, respectively. For $S_{\eta}$, an application of Fubini's theorem yields for $z>0$ that
\begin{align*}
\int_{0+}^{\infty}|S_{\eta}(t)|\dx t&\leq\int_{0+}^{\infty}\int_z^{\infty}|y-z|\nu_{\eta}\big|_{[-1,1]}(\dx y)\dx z=\int_{0+}^{\infty}\int_{0+}^y|y-z|\dx z\nu_{\eta}|_{[-1,1]}(\dx y)\\
&=\int_{0+}^{\infty}\frac{y^2}{2}\nu_{\eta}|_{[-1,1]}(\dx y)<\infty,
\end{align*}
and similarly for $z<0$, showing that $S_{\eta}$ is indeed integrable. The same argument applies to $z\mapsto S_{\widetilde{U}}(z+1)$. The remaining assertions now follow from standard results on the convolution of bounded or measurable functions and finite measures.
\end{proof}

The term involving $S_{\widetilde{U}}$ in the distributional equation~\eqref{eq-mu} below is considered in the following lemma.

\begin{lemma}\label{lem-functions2}
Let $q\geq0$ and $S_{\widetilde{U}}$ as defined in~\eqref{def-S-U}. Then
\begin{align*}
\rho(\dx z)&=\Big(\I_{\{z>0\}}\int_0^{\infty}xS_{\widetilde{U}}(\tfrac{z}{x})\mu(\dx x)+\I_{\{z<0\}}\int_{-\infty}^0|x|S_{\widetilde{U}}(\tfrac{z}{x})\mu(\dx x)\Big)\dx z,
\end{align*}
defines a locally finite measure on~$\cB(\RR)$.
\end{lemma}

\begin{proof}
Let $B\subset\RR$ be compact. We first consider $B\in[0,\infty)$, i.e. $B\subseteq[0,R]$ for sufficiently large $R\in\RR$. As $S_{\widetilde{U}}$ is nonnegative by definition, we obtain
\begin{displaymath}
\int_B\rho(\dx z)\leq\int_{0+}^R\int_0^{z-}xS_{\widetilde{U}}(\tfrac{z}{x})\mu(\dx x)\dx z+\int_{0+}^R\int_{z+}^{\infty}xS_{\widetilde{U}}(\tfrac{z}{x})\mu(\dx x)\dx z,
\end{displaymath}
in which we can insert the cases given in~\eqref{def-S-U}. Applying Fubini's theorem now yields
\begin{align*}
\int_{0+}^R\int_0^{z-}xS_{\widetilde{U}}(\tfrac{z}{x})\mu(\dx x)\dx z&\leq\int_{0+}^1\int_{0+}^R\int_x^{x+xy}(xy-z+x)\dx z\mu(\dx x)\nu_{\widetilde{U}}|_{[-1,1]}(\dx y)\\
&\leq\frac{R^2}{2}\int_{0+}^1y^2\nu_{\widetilde{U}}(\dx y)<\infty
\end{align*}
for the first term. For the second term, write
\begin{displaymath}
\int_{0+}^R\int_{z+}^{\infty}xS_{\widetilde{U}}(\tfrac{t}{x})\mu(\dx x)\dx z=\int_{-1}^{0-}\int_{0+}^{\infty}\int_{\min\{x(1+y),R\}}^{\min\{x,R\}}(z-x-xy)\dx z\mu(\dx x)\nu_{\widetilde{U}}(\dx y).
\end{displaymath}
Whenever $y$ is bounded away from zero, e.g. considering $\smash{y\in[-1,-1/2]}$, the inner intergral can again be estimated by $\smash{\int_0^Rz\dx z}=R^2/2$, thus yielding finiteness of the triple integral as before. For $y\in(-1/2,0)$, observe that the inner integral vanishes whenever $x>2R$ and if~$x\leq2R$, it can be bounded by $\smash{\int_{x(1+y)}^x(z-x-xy)\dx z}=x^2y^2/2\leq2y^2R^2$. Thus, the triple integral is also finite in the last case. Since the same arguments apply for $B\subset(-\infty,0]$, it follows that $\rho$ is locally finite.
\end{proof}

%
%%% Presentation of Functional Equation
%
The following theorem is the main result of this section. As before, we set $\cL(V_{q,\xi,\eta})=\mu$.
\begin{theorem}\label{main}
Let $\xi,\eta$ be two independent L\'{e}vy processes such that $\eta$ is not the zero process and $q\geq0$ such that $V_{0,\xi,\eta}$ converges a.s. whenever $q=0$ is considered. Further, let the functions~$B_{\eta},B_{\widetilde{U}},S_{\eta},S_{\widetilde{U}}$ be as in Lemma~\ref{lem-functions}. Then there exists a constant $K\in\RR$ such that
\begin{align}
K\dx z&=\Bigl(\frac{1}{2}\sigma_{\eta}^2+\frac{1}{2}z^2\sigma_{\widetilde{U}}^2\Bigr)\mu(\dx z)+(S_{\eta}\ast\mu)(z)\dx z\nonumber\\
&\quad+\Big(\I_{\{z>0\}}\int_0^{\infty}xS_{\widetilde{U}}(\tfrac{z}{x})\mu(\dx x)+\I_{\{z<0\}}\int_{-\infty}^0|x|S_{\widetilde{U}}(\tfrac{z}{x})\mu(\dx x)\Big)\dx z\nonumber\\
&\quad-\int_{0+}^z\Bigl(\gamma_{\eta}+x\gamma_{\widetilde{U}}\Bigr)\mu(\dx x)\dx z-\int_{0+}^z(B_{\eta}\ast\mu)(x)\dx x\dx z\nonumber\\
&\quad-\int_{0+}^z\int_{0+}^tB_{\widetilde{U}}(\tfrac{t}{x})\mu(\dx x)\dx t\dx z.\label{eq-mu}
\end{align}
\end{theorem}

The proof of Theorem~\ref{main} is based on the proof of Theorem~2.2 in~\cite{KuznetsovPardoSavov2012} and the individual steps are carried out in Section~\ref{s6} below. We sketch the argument briefly. First, taking $f\in C_c^{\infty}(\RR)$, the explicit form of~$\mathcal{A}^{\widetilde{V}}f(x)$ is inserted into \eqref{eq-invmeasure}, allowing to rewrite the left-hand side to the form
\begin{displaymath}
\int_{\RR}A^{\widetilde{V}}f(x)\mu(\dx x)=\int_{\RR}f''(z)G_1(\dx z)+\int_{\RR}f'(z)G_2(\dx z).
\end{displaymath}
for suitable $G_1$ and $G_2$. We can then use partial integration to rewrite the above integrals to all include the same function, namely $f''$, yielding the form
\begin{displaymath}
\int_{\RR}f''(z)G_1(\dx z)+\int_{\RR}f'(z)G_2(\dx z)=\int_{\RR}f''(z)G(\dx z),
\end{displaymath}
where $G$ can be identified with a distribution in the sense of Schwartz. Using \eqref{eq-invmeasure} and the definition of the distributional derivative, it follows that this distribution satisfies $G''=0$. By solving this ordinary differential equation (ODE) over the distribution space, one can find an equivalent expression for $G$. Identifying the remaining constants then yields the desired equation.

%
%%% Corollaries for Absolute Continuity + Density
%
Whenever $\mu$ is absolutely continuous with respect to the Lebesgue measure,~\eqref{eq-mu} directly yields a functional equation for the density. Various sufficient conditions for absolute continuity are given in Theorems~6.8 and~6.14 of~\cite{BLRRPart1}. Note in particular that whenever $\mu$ is continuous, the existence of a density is equivalent to the existence of a density of $\mu|_{\RR\setminus\{0\}}$. If $\eta$ is not the zero process, $\mu$ is continuous if and only if $q>0$ and $\eta$ is not a compound Poisson process (cf.~\cite[Cor.~6.11]{BLRRPart1}) or if $q=0$ and $\xi$ and $\eta$ are not simultanously deterministic (cf.~\cite[Thm. 2.2]{BertoinLindnerMaller2008}). In the case that $q>0$ and $\eta$ is a compound Poisson process, it is $\mu(\{0\})>0$ such that the measure cannot be absolutely continuous, however, it is still possible for $\mu|_{\RR\setminus\{0\}}$ to have a density (see Corollary~\ref{existence-density} below). We thus formulate the following result in the slightly more general setting that only $\mu|_{\RR\setminus\{0\}}$ has a density. The proof is immediate from Theorem~\ref{main} and is therefore omitted.

\begin{corollary}\label{corollary-equation-density}
Under the conditions of Theorem~\ref{main}, assume that $\mu|_{\RR\setminus\{0\}}$ has a density~$f_{\mu}$ with respect to the Lebesue measure. Then there exists a constant $K\in\RR$ such that
\begin{align}
&\Bigl(\frac{1}{2}\sigma_{\eta}^2+\frac{1}{2}z^2\sigma_{\widetilde{U}}^2\Bigr)f_{\mu}(z)+(S_{\eta}\ast f_{\mu})(z)+S_{\eta}(z)\mu(\{0\})\nonumber\\
&\quad+\I_{\{z>0\}}\int_0^{\infty}xS_{\widetilde{U}}(\tfrac{z}{x})f_{\mu}(x)\dx x+\I_{\{z<0\}}\int_{-\infty}^0|x|S_{\widetilde{U}}(\tfrac{z}{x})f_{\mu}(x)\dx x\nonumber\\
&=K+\int_0^z\Bigl(\gamma_{\eta}+x\gamma_{\widetilde{U}}\Bigr)f_{\mu}(x)\dx x-\I_{\{z<0\}}\gamma_{\eta}\mu(\{0\})\nonumber\\
&\quad+\int_0^z(B_{\eta}\ast f_{\mu})(x)\dx x+\int_0^z\int_0^tB_{\widetilde{U}}(\tfrac{t}{x})f_{\mu}(x)\dx x\dx t\label{equation-density}
\end{align}
for $\lambda$-a.e. $z\in\RR$.
\end{corollary}

It was shown in \cite[Cor.~2.5]{KuznetsovPardoSavov2012} that the law of the exponential functional~$V_{0,\xi,\eta}$ admits a continuous density on $\RR\setminus\{0\}$ if $\sigma_{\xi}^2+\sigma_{\eta}^2>0$, as well as $\ew|\xi_1|<\infty$, $\ew|\eta_1|<\infty$ and $\ew\xi_1<0$. The following Corollary generalizes this to general $q\geq0$. As in Theorem~\ref{main}, we do not require a moment condition. The proof is given in Section~\ref{s6}. Observe that $\smash{\sigma_{\widetilde{U}}^2=\sigma_{\xi}^2}$.

\begin{corollary}\label{existence-density}
In additions to the assumptions of Theorem~\ref{main}, assume that $\smash{\sigma_{\widetilde{U}}}^2+\sigma_{\eta}^2>0$.
\begin{itemize}
\item[(i)] If $\sigma_{\eta}^2>0$, then $\mu$ has a continuous density $f_{\mu}$ on $\RR$.
\item[(ii)] If $\sigma_{\widetilde{U}}^2>0$, then $\mu|_{\RR\setminus\{0\}}$ has a continuous density $f_{\mu}$ on $\RR\setminus\{0\}$.
\item[(iii)] In both cases, there exist constants $M_1,M_2>0$ such that
\begin{displaymath}
(\sigma_{\eta}^2+z^2\sigma_{\widetilde{U}}^2)f_{\mu}(z)\leq M_1+M_2|z|
\end{displaymath}
for all $z\neq0$.
\end{itemize}
\end{corollary}

Note that the above results, in particular~\eqref{eq-mu} and \eqref{equation-density}, are derived under very weak assumptions. Thus, the equations can be simplified further whenever more properties of the processes $\xi$ and~$\eta$ are known. We discuss some special cases in the following corollaries, the proofs of which are also given in Section~\ref{s6}.
%
%%% Discussion of Special Cases
%
\begin{corollary}[Finite First Moments]\label{corollary-finitemoments}
Under the assumptions of Theorem~\ref{main} let further~${\EE|\eta_1|<\infty}$ and $\EE|\widetilde{U}_1|<\infty$. Denote the expectation of $\eta_1$ and $\widetilde{U}_1$ by $\gamma_{\eta}^1$ and $\gamma_{\widetilde{U}}^1$, respectively, and define the functions
\begin{align*}
S_{\eta}^{FM}:\RR\rightarrow[0,\infty),\quad S_{\eta}^{FM}(z)&=\begin{cases}\int_{-\infty}^z(z-y)\nu_{\eta}(\dx y),\ &\text{if}\ z<0,\\0,&\text{if}\ z=0,\\ \int_z^{\infty}(y-z)\nu_{\eta}(\dx y), &\text{if}\ z>0,\\ \end{cases}\\
S_{\widetilde{U}}^{FM}:[0,\infty)\rightarrow[0,\infty),\quad S_{\widetilde{U}}^{FM}(z)&=\begin{cases}\int_{-\infty}^{z-1}(z-1-y)\nu_{\widetilde{U}}(\dx y),\ &\text{if}\ z\in[0,1),\\0,&\text{if}\ z=1,\\ \int_{z-1}^{\infty}(y-z+1)\nu_{\widetilde{U}}(\dx y),\ &\text{if}\ z>1.\end{cases}
\end{align*}
Then the following hold true:
\begin{itemize}
\item[(i)] There exists a constant $K\in\RR$ such that
\begin{align}
K\dx z&=\Bigl(\frac{1}{2}\sigma_{\eta}^2+\frac{1}{2}z^2\sigma_{\widetilde{U}}^2\Bigr)\mu(\dx z)+(S_{\eta}^{FM}\ast\mu)(z)\dx z\nonumber\\
&\quad+\Big(\I_{\{z>0\}}\int_0^{\infty}xS_{\widetilde{U}}^{FM}(\tfrac{z}{x})\mu(\dx x)+\I_{\{z<0\}}\int_{-\infty}^0|x|S_{\widetilde{U}}^{FM}(\tfrac{z}{x})\mu(\dx x)\Big)\dx z\nonumber\\
&\quad-\int_{0+}^z\Bigl(\gamma_{\eta}^1+x\gamma_{\widetilde{U}}^1\Bigr)\mu(\dx x)\dx z,\label{eq-mu-fm}
\end{align}
where the right-hand side of the equation defines a locally finite measure on~$\cB(\RR)$.
\item[(ii)] If additionally $\EE|\cE(U)_1|<\re^q$ or, equivalently, $\gamma_{\widetilde{U}}^1<0$, then~${\EE |V_{q,\xi,\eta}|=\int |x|\mu(\dx x)<\infty}$ and the constant $K$ in~\eqref{eq-mu-fm} takes the form
\begin{displaymath}
K=-\int_{0+}^{\infty}\big(\gamma_{\eta}^1+x\gamma_{\widetilde{U}}^1\big)\mu(\dx x)=\int_{-\infty}^0\big(\gamma_{\eta}^1+x\gamma_{\widetilde{U}}^1\big)\mu(\dx x).
\end{displaymath}
Moreover, if additionally $\sigma_{\eta}^2+\sigma_{\widetilde{U}}^2>0$, then the density $f_{\mu}$ of $\mu|_{\RR\setminus\{0\}}$ is bounded.
\end{itemize}
\end{corollary}

\begin{corollary}[Finite Variation]\label{corollary-finitevariation}
Under the assumptions of Theorem~\ref{main} let further $\eta$ and~$\smash{\widetilde{U}}$ be of finite variation, i.e. $\sigma_{\eta}^2=\sigma_{\widetilde{U}}^2=0$ and $\smash{\int_{[-1,1]}|x|\nu_{\eta}(\dx x),\int_{[-1,1]}|x|\nu_U(\dx x)}<\infty$. Denote by $\gamma_{\eta}^0$ and $\gamma_{\widetilde{U}}^0$ the drifts of $\eta$ and $\smash{\widetilde{U}}$, respectively, and define the functions
\begin{align}
B^{FV}_{\eta}:\RR\rightarrow\RR,\quad B^{FV}_{\eta}(z)&=\begin{cases} -\nu_{\eta}((-\infty,z)),\ &\text{if}\ z<0,\\ 0,&\text{if}\ z=0,\\\nu_{\eta}((z,\infty)), &\text{if}\ z>0,\end{cases}\label{def-B-eta-FV}\\
B^{FV}_{\widetilde{U}}:[0,\infty)\rightarrow\RR,\quad B^{FV}_{\widetilde{U}}(z)&=\begin{cases}-\nu_{\widetilde{U}}((-\infty,z-1)),\ &\text{if}\ z\in[0,1),\\0,&\text{if}\ z=1,\\ \nu_{\widetilde{U}}((z-1,\infty)),\ &\text{if}\ z>1.\end{cases}\label{def-B-U-FV}
\end{align}
Then the equation
\begin{align}
0&=\bigl(\gamma_{\eta}^0+z\gamma_{\widetilde{U}}^0\bigr)\mu(\dx z)+(B_{\eta}^{FV}\ast\mu)(z)\dx z\nonumber\\
&\quad+\Big(\I_{\{z>0\}}\int_0^{\infty}B_{\widetilde{U}}^{FV}(\tfrac{z}{x})\mu(\dx x)-\I_{\{z<0\}}\int_{-\infty}^0B_{\widetilde{U}}^{FV}(\tfrac{z}{x})\mu(\dx x)\Big)\dx z\label{eq-mu-fv}
\end{align}
holds and the quantities on the right-hand side define a locally finite measure.
\end{corollary}

Assuming finite variation only of the jump parts of the processes $\eta$ and $\smash{\widetilde{U}}$, Corollary~\ref{existence-density} can be extended to differentiability.
\begin{corollary}[Differentiable Density]\label{corollary-diffdensity}
Under the assumptions of Theorem~\ref{main}, let the jump parts of $\widetilde{U}$ and $\eta$ be of finite variation. Further, let~$\sigma_{\eta}^2+\sigma_{\widetilde{U}}^2>0$ and~$\gamma_{\eta}^0$, $\gamma_{\widetilde{U}}^0$, $B_{\eta}^{FV}$ and~$B_{\widetilde{U}}^{FV}$ as in Corollary~\ref{corollary-finitevariation}.
\begin{itemize}
\item[(i)] If $\sigma_{\eta}^2>0$, then the density $f_{\mu}$ of $\mu$ is continuously differentiable on $\RR\setminus\{0\}$.
\item[(ii)] If $\sigma_{\widetilde{U}}^2>0=\sigma_{\eta}^2$  and $q=0$, or $\eta$ is not a compound Poisson process, then $\mu$ has a density $f_{\mu}$ on $\RR$ which is continuously differentiable on $\RR\setminus\{0\}$.
\item[(iii)] The density $f_{\mu}$ satisfies the equation
\begin{align}
&\Big(\frac{1}{2}\sigma_{\eta}^2+\frac{1}{2}z^2\sigma_{\widetilde{U}}^2\Big)f'_{\mu}(z)+z\sigma_{\widetilde{U}}^2f_{\mu}(z)-\bigl(\gamma_{\eta}^0+z\gamma_{\widetilde{U}}^0\bigr)f_{\mu}(z)-B_{\eta}^{FV}(z)\mu(\{0\})\label{eq-mu-diffdensity}\\
&\quad=(B_{\eta}^{FV}\ast f_{\mu})(z)+\I_{\{z>0\}}\int_0^{\infty}B_{\widetilde{U}}^{FV}(\tfrac{z}{x})f_{\mu}(x)\dx x-\I_{\{z<0\}}\int_{-\infty}^0B_{\widetilde{U}}^{FV}(\tfrac{z}{x})f_{\mu}(x)\dx x,\nonumber
\end{align}
which under the conditions (i) and (ii) is valid for all $z\in\RR\setminus\{0\}$ and still holds $\lambda$-a.e. whenever $\smash{\sigma_{\widetilde{U}}^2>0=\sigma_{\eta}^2}$, but the additional assumptions of (ii) are not satisfied. In the latter case, $f_{\mu}$ is $\lambda$-a.e. differentiable.
\end{itemize}
\end{corollary}
Observe that $\mu(\{0\})=0$ when $\sigma_{\eta}^2>0$, or $\sigma_{\eta}^2+\sigma_{\widetilde{U}}^2>0$ and $q=0$, or $q>0$, $\sigma_{\widetilde{U}}^2>0$ and $\eta$ is neither a compound Poisson process nor the zero process.
\begin{remark}\label{rem-literature}\rm
Using the relation between the characteristic triplets of $\xi$, $U$, and $\smash{\widetilde{U}}$ established in Section~\ref{s2} and  Proposition~\ref{prop-sdewithkilling}, as well as the fact that $\gamma_{\widetilde{U}}^0=\gamma_U^0=-\gamma_{\xi}^0+\frac{1}{2}\sigma_{\xi}^2$ whenever~$\smash{\int_{[-1,1]}|y|\nu_{\xi}(\dx y)}<\infty$, one finds that Proposition~\ref{prop-dens-form1} is a special case of Corollaries~\ref{corollary-finitevariation} and~\ref{corollary-diffdensity}. In particular, Equation~\eqref{eq-densitylaplace} is reobtained from~\eqref{eq-mu-fv} and~\eqref{eq-mu-diffdensity} for~$z>0$. If $\eta$ is a subordinator, similar formulas are obtained from~\eqref{eq-mu-fv} and~\eqref{eq-mu-diffdensity} for $z<0$, which are readily seen to be satisfied by $f_{\mu}(z)=0$ for $z<0$. Note, however, that neither of the corollaries requires $\xi$ or $\eta$ to be a subordinator.
\end{remark}

\begin{remark}\label{validity-analysis}\rm{
While studying the method, we found that the distributional equation given in~(2.3) of~\cite{KuznetsovPardoSavov2012} for the case $q=0$ and $\ew|\xi_1|,\ew|\eta_1|<\infty$ does not hold in general. The cause of this lies in equation~(2.6) of the paper, where it is stated that for functions~${f\in C_c^{\infty}((0,\infty))}$ the left-hand side of the equation $\int_{\RR}A^Vf(x)\mu(\dx x)=0$ simplifies to an integral over the positive real line, i.e.~$\int_0^{\infty}\cA^Vf(x)\mu(\dx x)=0$. Evaluating the generator for such a function $f$ leads to
\begin{align}\label{remainder-term}
\int_{\RR}\mathcal{A}^Vf(x)\mu(\dx x)&=\int_0^{\infty}\mathcal{A}^Vf(x)\mu(\dx x)\nonumber\\
&\quad+\int_{-\infty}^0\frac{\sigma_{\eta}^2}{2}f''(x)+(\gamma_{\eta}-x\gamma_{\xi})f'(x)+\frac{\sigma_{\xi}^2}{2}(x^2f''(x)+xf'(x))\mu(\dx x)\nonumber\\
&\quad+\int_{-\infty}^0\int_{-\infty}^0\big(f(x+y)-f(x)-yf'(x)\I_{\{|y|\leq1\}}\big)\nu_{\eta}(\dx y)\mu(\dx x)\nonumber\\
&\quad+\int_{-\infty}^0\int_0^{\infty}\big(f(x+y)-f(x)-yf'(x)\I_{\{|y|\leq1\}}\big)\nu_{\eta}(\dx y)\mu(\dx x)\nonumber\\
&\quad+\int_{-\infty}^0\int_{\RR}\big(f(x\re^{-y})-f(x)+f'(x)xy\I_{\{|y|\leq1\}}\big)\nu_{\xi}(\dx y)\mu(\dx x).
\end{align}
Observe that the second, third and last term are zero as $f(x)=0$ for $x\leq0$. However, the fourth term may not, e.g. in the case $\xi_t=t$ and $\eta$ being a pure-jump process with L\'{e}vy measure~$\nu_{\eta}=\delta_2+\delta_{-2}$. For this example, one can construct a nonnegative test function supported on the interval $[\frac{1}{2},\frac{3}{2}]$ for which the term in question is nonzero. Nevertheless, whenever $\eta$ is a subordinator, and thus $V_{0,\xi,\eta}\geq0$ a.s., or $\eta$ does not have any positive jumps, the two last terms of \eqref{remainder-term} vanish such that all conclusions drawn from equation~(2.6) in~\cite{KuznetsovPardoSavov2012}, in particular the distributional equation~(2.3), remain valid. Otherwise, the equation does not necessarily hold, as can be seen from the following example. Let~${\xi_t=t}$ and $\eta$ be a pure-jump process with the L\'{e}vy measure given by~${\nu_{\eta}(\dx x)=\re^{-|x|}\dx x}$. We can derive the distribution of $V_{0,\xi,\eta}$ explicitly from~\cite[Thm.~2.1(f)]{GjessingPaulsen1997}, yielding that the exponential functional has the same distribution as the difference of two independent Exp(1)-distributed random variables, i.e. a Laplace distribution with parameters 0 and 1. As $\mu$ is known, one can readily check that Equation~(2.3) of~\cite{KuznetsovPardoSavov2012} does not hold for this example. The tail function and the integrated tails for~${x>0}$ are given by
\begin{align*}
\nu_{\eta}((x,\infty))=\int_x^{\infty}\re^{-t}\dx t=\re^{-x},\quad \inttail{\eta}{+}(x):=\int_x^{\infty}\nu_{\eta}((t,\infty))\dx t=\re^{-x},
\end{align*}
and similarly
\begin{displaymath}
\inttail{\eta}{-}(x):=\int_x^{\infty}\nu_{\eta}((-\infty,-t))\dx t=\re^{-x},\quad x>0.
\end{displaymath}
Therefore, Equation (2.3) in \cite{KuznetsovPardoSavov2012} reads
\begin{align}\label{eq2.3here}
&-\int_v^{\infty}\mu(\dx x)\dx v+\Bigl(\frac{1}{v}\int_0^v\re^{-(v-x)}\mu(\dx x)\Bigr)\dx v+\Bigl(\frac{1}{v}\int_v^{\infty}\re^{-(x-v)}\mu(\dx x)\Bigr)\dx v\nonumber\\
&\quad-\int_v^{\infty}\frac{1}{w^2}\Bigl(\int_0^w e^{-(w-x)}\mu(\dx x)+\int_w^{\infty}\re^{-(x-w)}\mu(\dx x)\Bigr)\dx w\dx v=0,\quad v>0
\end{align}
for this specific example. Note that due to the choice of the processes the remaining parameters (in the notation of~\cite{KuznetsovPardoSavov2012}) are given by $b_{\xi}=-1$, $\sigma_{\xi}^2=\sigma_{\eta}^2=0$ and $\smash{\inttail{\xi}{+}=\inttail{\xi}{-}=0}$. Inserting ${\mu(\dx x)=\frac{1}{2}\re^{-|x|}\dx x}$ into the left-hand side of~\eqref{eq2.3here}, we obtain
\begin{align*}
&-\frac{\re^{-v}}{2}\dx v+\frac{\re^{-v}}{2}\dx v+\frac{\re^{-v}}{4v}\dx v-\frac{1}{4}\Big(\int_v^{\infty}\frac{2\re^{-w}}{w}\dx w+\frac{\re^{-v}}{v}-\int_v^{\infty}\frac{\re^{-w}}{w}\dx w\Big)\dx v\\
&=\Big(-\frac{1}{4}\int_v^{\infty}\frac{\re^{-w}}{w}\dx w\Big)\dx v,
\end{align*}
which is not the zero measure, contradicting~\eqref{eq2.3here}. However, one can check that the equation given in~\eqref{eq-mu-fv} is satisfied for this example. Observing that $U_t=\smash{\widetilde{U}_t}=-t$ here and verifying that $\eta$ and $\smash{\widetilde{U}}$ are of finite variation, Corollary~\ref{corollary-finitevariation} is applicable. Hence, it holds
\begin{equation}\label{eq5.8here}
0=-z\mu(\dx z)+(B_{\eta}^{FV}\ast\mu)(z)\dx z,
\end{equation}
by~\eqref{eq-mu-fv}, where $B_{\eta}^{FV}$ can be calculated explicitly as
\begin{align*}
B_{\eta}^{FV}(z)=\begin{cases}\re^{-z},\ z>0\\-\re^z,\ z<0\end{cases}\negmedspace\negmedspace\negmedspace\negmedspace\negmedspace\Bigg\}=\mathrm{sign}(z)\re^{-|z|}.
\end{align*}
Therefore,~\eqref{eq5.8here} now reads
\begin{displaymath}
z\mu(\dx z)=\Big(\int_{-\infty}^z\re^{-z+s}\mu(\dx s)-\int_z^{\infty}\re^{z-s}\mu(\dx s)\Big)\dx z
\end{displaymath}
and it is readily checked that the equation indeed holds for $\mu(\dx x)=\frac{1}{2}\re^{-|x|}\dx x$. 

Instead of using the results from~\cite{GjessingPaulsen1997}, one could also solve~\eqref{eq5.8here} directly. Since $B_{\eta}^{FV}\ast\mu$ is integrable, so is $z\mu(\dx z)$ by~\eqref{eq5.8here}, such that taking Fourier transforms leads to
\begin{displaymath}
-i\phi_{V_{0,\xi,\eta}}'(x)=-\frac{2x}{x^2+1}\phi_{V_{0,\xi,\eta}}(x).
\end{displaymath}
This yields $\smash{\phi_{V_{0,\xi,\eta}}}(x)=(x^2+1)^{-1}$, from which the exact distribution of $V_{0,\xi,\eta}$ is readily obtained by Fourier inversion. Alternatively, one could also observe that $\EE V_{0,\xi,\eta}^2<\infty$ by~\eqref{eq-condsecondmoment} and find the distribution of $V_{0,\xi,\eta}$ from solving~\eqref{eq-diffeqchar}, or equivalently (4.8) in~\cite{BehmeLindner2015}, for the given characteristics and performing a Fourier inversion of the solution.}
\end{remark}

%
%%% PROOFS
%

\section{Proofs for Section~\ref{s5}}\label{s6}
The proof of Theorem~\ref{main} consists of several steps which are shown as separate lemmas. First, the left-hand side of~\eqref{eq-invmeasure} is rewritten to a suitable form.

%
%%% Lemmas for Rewriting
%
\begin{lemma}\label{rewriting1}
Under the assumptions of Theorem \ref{main} we have for every $f\in C_c^{\infty}(\RR)$ that
\begin{displaymath}
\int_{\RR}A^{\widetilde{V}}f(x)\mu(\dx x)=\int_{\RR}f''(z)G_1(\dx z)+\int_{\RR}f'(z)G_2(\dx z),
\end{displaymath}
where the individual contributions are given by
\begin{align*}
G_1(\dx z)&=\Bigl(\frac{1}{2}\sigma_{\eta}^2+\frac{1}{2}z^2\sigma_{\widetilde{U}}^2\Bigr)\mu(\dx z)+(S_{\eta}\ast\mu)(z)\dx z\\
&\quad+\Big(\I_{\{z>0\}}\int_0^{\infty}xS_{\widetilde{U}}(\tfrac{z}{x})\mu(\dx x)+\I_{\{z<0\}}\int_{-\infty}^0|x|S_{\widetilde{U}}(\tfrac{z}{x})\mu(\dx x)\Big)\dx z\\
G_2(\dx z)&=\bigl(\gamma_{\eta}+z\gamma_{\widetilde{U}}\bigr)\mu(\dx z)+(B_{\eta}\ast\mu)(z)\dx z\\
&\quad+\int_{0+}^zB_{\widetilde{U}}(\tfrac{z}{x})\mu(\dx x)\dx z
\end{align*}
with the functions $B_{\eta}$, $B_{\widetilde{U}}$, $S_{\eta}$, and $S_{\widetilde{U}}$ given as in Equations~\eqref{def-B-eta} to \eqref{def-S-U}.
\end{lemma}

\begin{proof}
By linearity, we can split $\mathcal{A}^{\widetilde{V}}f(x)$ according to~\eqref{eq-genkilledgou} and rewrite the corresponding integrals separately. Firstly, for the terms originating from the Gaussian and drift components it follows that
\begin{align*}
&\int_{\RR}\Bigl(\frac{1}{2}\sigma_{\eta}^2f''(x)+\gamma_{\eta}f'(x)+\frac{1}{2}x^2f''(x)\sigma_{\widetilde{U}}^2+xf'(x)\gamma_{\widetilde{U}}\Bigr)\mu(\dx x)\\
&\quad=\int_{\RR}f''(x)\Bigl(\frac{1}{2}\sigma_{\eta}^2+\frac{1}{2}x^2\sigma_{\widetilde{U}}^2\Bigr)\mu(\dx x)+\int_{\RR}f'(x)\Bigl(\gamma_{\eta}+x\gamma_{\widetilde{U}}\Bigr)\mu(\dx x)
\end{align*}
such that their contributions to $G_1$ and $G_2$ are readily identified. For the terms corresponding to the jump parts of the processes, the integrals with respect to the L\'{e}vy measure are split according to the value of the indicator function in the integrand. Starting with the contribution of the big jumps of $\eta$, we find for $y>1$ that
\begin{align*}
\int_{\RR}\int_{1+}^{\infty}\big(f(x+y)-f(x)\big)\nu_{\eta}(\dx y)\mu(\dx x)&=\int_{\RR}\int_{1+}^{\infty}\int_x^{x+y}f'(t)\dx t\nu_{\eta}(\dx y)\mu(\dx x)\\
&=\int_{\RR}f'(t)\int_{-\infty}^t\nu_{\eta}\big((\max\{t-x,1\},\infty)\big)\mu(\dx x)\dx t,
\end{align*}
where interchanging the order of integration is allowed due to the compact support of $f'$ and the involved measures being finite. A similar calculation applies if $y<-1$. Using the function $B_{\eta}$ defined in~\eqref{def-B-eta}, the term reads as
\begin{align*}
\int_{\RR}\int_{\RR\setminus[-1,1]}\big(f(x+y)-f(x)\big)\nu_{\eta}(\dx y)\mu(\dx x)&=\int_{\RR}f'(t)\int_{\RR}B_{\eta}(t-x)\mu(\dx x)\dx t\\
&=\int_{\RR}f'(t)(B_{\eta}\ast\mu)(t)\dx t.
\end{align*}
The big jumps of $\smash{{\widetilde{U}}}$ are treated in the same way, although the result cannot be interpreted as a linear convolution here. For $x>0$ it follows that
\begin{align*}
\int_{0+}^{\infty}\int_{1+}^{\infty}\big(f(x+xy)-f(x)\big)\nu_{\widetilde{U}}(\dx y)\mu(\dx x)=\int_0^{\infty}f'(t)\int_{0+}^t\nu_{\widetilde{U}}\big((\max\{\tfrac{t}{x}-1,1\},\infty)\big)\mu(\dx x)\dx t,
\end{align*}
and the calculation for $x<0$ is again similar. Using the function $B_{\widetilde{U}}$ introduced in~\eqref{def-B-U} now yields the desired form as
\begin{align*}
&\int_{\RR}\int_{1+}^{\infty}\big(f(x+xy)-f(x)\big)\nu_{\widetilde{U}}(\dx y)\mu(\dx x)=\int_{\RR}f'(t)\int_{0+}^tB_{\widetilde{U}}(\tfrac{t}{x})\mu(\dx x)\dx t.
\end{align*}
Note that the argument of $B_{\widetilde{U}}$ is always greater or equal to one due to $t$ and $x$ being of the same sign with $|x|\leq|t|$. The approach to the terms corresponding to the small jumps of $\eta$ and $\smash{\widetilde{U}}$, respectively, is similar. However, we obtain a contribution to~$G_1$ instead of~$G_2$ here. For $\eta$, using the Taylor formula, this leads to
\begin{align*}
&\int_{\RR}\int_{[-1,1]}\big(f(x+y)-f(x)-yf'(x)\big)\nu_{\eta}(\dx y)\mu(\dx x)\\
&\quad=\int_{\RR}\int_{[-1,1]}\int_x^{x+y}f''(t)(x+y-t)\dx t\nu_{\eta}(\dx y)\mu(\dx x).
\end{align*}
A direct computation similar to Lemma~\ref{lem-functions} shows, since $|f''|$ is compactly supported and thus bounded by a constant, that
\begin{align*}
\Big|\int_x^{x+y}|f''(t)(x+y-t)|\dx t\Big|\leq\frac{Cy^2}{2}.
\end{align*}
Thus, Fubini's theorem is applicable and we find for $y>0$ that
\begin{align*}
&\int_{\RR}\int_{(0,1]}\big(f(x+y)-f(x)-yf'(x)\big)\nu_{\eta}(\dx y)\mu(\dx x)\\
&\quad=\int_{\RR}f''(t)\int_{-\infty}^t\int_{t-x}^{\infty}\big(y-(t-x)\big)\nu_{\eta}\big|_{[-1,1]}(\dx y)\mu(\dx x)\dx t,
\end{align*}
with a similar calculation holding for $y<0$. Adding both terms, one obtains
\begin{align*}
&\int_{\RR}\int_{[-1,1]}\big(f(x+y)-f(x)-yf'(x)\big)\nu_{\eta}(\dx y)\mu(\dx x)=\int_{\RR}f''(t)(S_{\eta}\ast\mu)(t)\dx t,
\end{align*}
where the function $S_{\eta}$ is taken from~\eqref{def-S-eta}. For the last term involving the small jumps of~$\smash{\widetilde{U}}$, it follows similarly that
\begin{align*}
&\int_{\RR}\int_{[-1,1]}\big(f(x+xy)-f(x)-xyf'(x)\big)\nu_{\widetilde{U}}(\dx y)\mu(\dx x)\\&\quad=\int_{\RR}\int_{[-1,1]}\int_x^{x+xy}f''(t)(x+xy-t)\dx t\nu_{\widetilde{U}}(\dx y)\mu(\dx x).
\end{align*}
As $f$ has compact support, there is some~$R>0$ such that $\mathrm{supp}(f)\subseteq[-R,R]$. Let now~$x,y>0$ and denote the set ${\mathrm{supp}(f)\cap[x,x+xy]}$ by $M=M_{x,y}$. As $M=\emptyset$ if $x>R$ and~$|f''|$ is bounded by some constant $C$, it follows that
\begin{align*}
&\int_{0+}^1\int_0^{\infty}\int_x^{x+xy}|f''(t)(x+xy-t)|\dx t\mu(\dx x)\nu_{\widetilde{U}}(\dx y)\leq \frac{CR^2}{2}\int_{(0,1]}y^2\nu_{\widetilde{U}}(\dx y)<\infty
\end{align*}
with a similar calculation as in Lemma~\ref{lem-functions2}. When considering $x<0$, $R$ is replaced by~$-R$. If $y<0$ and $x>0$, we can split the interval $[-1,0)$ at some intermediate point~$y_0$, say~$y_0=\frac{1}{2}$, and estimate the respective integrals separately. For $y\in(-\frac{1}{2},0)$, observe that~$M=\mathrm{supp}(f)\cap[x+xy,x]=\emptyset$ if $x>2R$ as $x+xy>\frac{x}{2}$ for the given values of $y$. Thus, we can use similar estimates as for $y>0$. For $y\in[-1,-\frac{1}{2}]$, note that $M\subseteq[0,R]$ and that, therefore, $x(1+y)\in M$ only if $x(1+y)\leq R$. This yields
\begin{align*}
&C\int_{[-1,-\frac{1}{2}]}\int_0^{\infty}\int_M(t-x-xy)\dx t\mu(\dx x)\nu_{\widetilde{U}}(\dx y)\leq \frac{C R^2}{2}\nu_{\widetilde{U}}([-1,-\tfrac{1}{2}])<\infty,
\end{align*}
due to $[-1,-\frac{1}{2}]$ being bounded away from zero. Again, similar arguments are applicable for negative values of $x$, i.e. when~$y<0$ and~$x<0$ yielding integrability in the last case. Interchanging the order of integration and rewriting the term to include $S_{\widetilde{U}}$ as defined in~\eqref{def-S-U} now leads to
\begin{align*}
&\int_0^{\infty}\int_0^1\int_x^{x+xy}f''(t)(x+xy-t)\dx t\nu_{\widetilde{U}}(\dx y)\mu(\dx x)\\
& \quad\quad+\int_0^{\infty}\int_{-1}^0\int_{x+xy}^xf''(t)(t-x-xy)\dx t\nu_{\widetilde{U}}(\dx y)\mu(\dx x)\\
&\quad=\int_0^{\infty}f''(t)\int_0^{\infty}xS_{\widetilde{U}}(\tfrac{t}{x})\mu(\dx x)\dx t.
\end{align*}
As the remaining two terms yield a similar result with the opposite sign, the complete term can be rewritten as
\begin{align*}
&\int_{\RR}\int_{[-1,1]}\big(f(x+xy)-f(x)-xyf'(x)\big)\nu_{\widetilde{U}}(\dx y)\mu(\dx x)\\
&\quad=\int_{\RR}f''(t)\Big(\I_{\{t>0\}}\int_0^{\infty}xS_{\widetilde{U}}(\tfrac{t}{x})\mu(\dx x)-\I_{\{t<0\}}\int_{-\infty}^0xS_{\widetilde{U}}(\tfrac{t}{x})\mu(\dx x)\Big)\dx t.
\end{align*}
Summing up the individual contributions now yields $G_1$ and $G_2$ as claimed.
\end{proof}

\begin{lemma}\label{rewriting2}
Under the assumptions of Theorem \ref{main} we have for all $f\in C_c^{\infty}(\RR)$ that
\begin{displaymath}
\int_{\RR}f''(z)G_1(\dx z)+\int_{\RR}f'(z)G_2(\dx z)=\int_{\RR}f''(z)G(\dx z),
\end{displaymath}
where $G$ can be identified with a distribution in the sense of Schwartz and is given by
\begin{align*}
G(\dx z)&=\Bigl(\frac{1}{2}\sigma_{\eta}^2+\frac{1}{2}z^2\sigma_{\widetilde{U}}^2\Bigr)\mu(\dx z)+(S_{\eta}\ast\mu)(z)\dx z\\
&\quad+\Big(\I_{\{z>0\}}\int_0^{\infty}xS_{\widetilde{U}}(\tfrac{z}{x})\mu(\dx x)+\I_{\{z<0\}}\int_{-\infty}^0|x|S_{\widetilde{U}}(\tfrac{z}{x})\mu(\dx x)\Big)\dx z\\
&\quad-\int_{0+}^z\big(\gamma_{\eta}+x\gamma_{\widetilde{U}}\big)\mu(\dx x)\dx z-\int_{0+}^z(B_{\eta}\ast\mu)(x)\dx x\dx z\\
&\quad-\int_{0+}^z\int_{0+}^tB_{\widetilde{U}}(\tfrac{t}{x})\mu(\dx x)\dx t\dx z,
\end{align*}
with the functions $S_{\eta}$, $S_{\widetilde{U}}$, $B_{\eta}$ and $B_{\widetilde{U}}$ as defined in Equations~\eqref{def-B-eta} to~\eqref{def-S-U}.
\end{lemma}

\begin{proof}
The term involving $G_1(\dx z)$ in Lemma~\ref{rewriting1} is already of the desired form, and, by Lemmas~\ref{lem-functions} and~\ref{lem-functions2}, it follows that $G_1(\dx z)$ yields finite values when evaluated over compact subsets of $\RR$. For the terms included in $G_2$, observe that $z\mapsto\int_{0+}^zG_2(\dx w)$ is càdlàg on $\RR$ and locally of bounded variation. Partial integration then shows
\begin{align*}
\int_{\RR}f'(z)G_2(\dx z)&=-\int_{\RR}f''(z)\int_{0+}^zG_2(\dx w)\dx z\\
&=-\int_{\RR}f''(z)\Big(\int_{0+}^z\bigl(\gamma_{\eta}+x\gamma_{\widetilde{U}}\bigr)\mu(\dx x)+\int_{0+}^z(B_{\eta}\ast\mu)(x)\dx x\\
&\quad+\int_{0+}^z\int_{0+}^tB_{\widetilde{U}}(\tfrac{t}{x})\mu(\dx x)\dx t\dx z\Bigr)\dx z,
\end{align*}
This contribution to $G$ also yields finite values when evaluated over compact subsets of $\RR$ by Lemma~\ref{lem-functions}. Summing up the terms, we find that $G$ is of the claimed form and locally finite, which allows to interpret the measure as a distribution in the sense of Schwartz.
\end{proof}

The following lemma now allows to identify the distribution $G$ through solving an ordinary differential equation.
\begin{lemma}\label{odesolve}
The distribution $G(\dx z)$ in Lemma \ref{rewriting2} is of the form $C_1z\dx z+C_2\dx z$ for some constants $C_1,C_2\in\RR$.
\end{lemma}

\begin{proof}
By Equation \eqref{eq-invmeasure} and Lemma \ref{rewriting2} it holds for all $f\in C_c^{\infty}(\RR)$ that
\begin{displaymath}
\int_{\RR}A^{\widetilde{V}}f(x)\mu(\dx x)=\int_{\RR}f''(z)G(\dx z)=\langle f'',G\rangle=0,
\end{displaymath}
where $\langle\cdot,\cdot\rangle$ denotes dual pairing. From the definition of the distributional derivative it now follows that
\begin{displaymath}
\langle f'',G\rangle=-\langle f',G'\rangle=\langle f,G''\rangle=0.
\end{displaymath}
As the above holds for all test functions $f$ and $\RR$ is an open set, we can conclude that~$G''$ must be the zero distribution. Using results on the antiderivative of distributions, e.g. from~\cite[Thm.~4.3]{DuistermaatKolk2010}, we find that the solution is given by~${G(\dx z)=C_1z\dx z+C_2\dx z}$ and is unique up to the choice of constants.
\end{proof}

Lastly, we note the following lemma to identify one of the constants.
\begin{lemma}\label{constants}
The distribution $G(\dx z)$ in Lemma~\ref{rewriting2} satisfies
\begin{equation}\label{limit-constants1}
\frac{1}{\ln(t)}\int_1^t\frac{1}{z^2}G(\dx z)\rightarrow0,\ t\rightarrow\infty.
\end{equation}
\end{lemma}

\begin{proof}
Using linarity, we can once more treat every summand of $G$ separately. First, we find for the contribution of the Gaussian parts of $\eta$ and $\smash{\widetilde{U}}$ that
\begin{align*}
\frac{1}{\ln(t)}\int_1^t\Big(\frac{1}{z^2}\frac{1}{2}\sigma_{\eta}^2+\frac{1}{2}\sigma_{\widetilde{U}}^2\Big)\mu(\dx x)\leq\frac{1}{\ln(t)}\Big(\frac{1}{2}\sigma_{\eta}^2+\frac{1}{2}\sigma_{\widetilde{U}}^2\Big)\mu([1,t])\rightarrow0,\ t\rightarrow\infty,
\end{align*}
yielding the desired value of the limit as $\mu$ is a finite measure. For the contribution of  the drift first observe that
\begin{align*}
\lim_{z\rightarrow\infty}\frac{1}{z}\int_0^zx\mu(\dx x)=0
\end{align*}
i.e. for every $\epsilon>0$ we can find a value $R_{\epsilon}$ such that $\frac{1}{z}\int_0^zx\mu(\dx x)\leq\epsilon$ if $z>R_{\epsilon}$. This yields
\begin{align*}
\Big|\int_1^t\frac{1}{z^2}\int_0^z\gamma_{\widetilde{U}}x\mu(\dx x)\dx z\Big|\leq|\gamma_{\widetilde{U}}|\Big(\int_1^{R_{\epsilon}}\frac{1}{z^2}\int_0^zx\mu(\dx x)\dx z+\epsilon\int_{R_{\epsilon}}^t\frac{1}{z}\dx z\Big)
\end{align*}
which implies that
\begin{align*}
0\leq\limsup_{t\rightarrow\infty}\Big|\frac{1}{\ln(t)}\int_1^t\frac{1}{z^2}\int_0^z\big(\gamma_{\eta}+x\gamma_{\widetilde{U}}\big)\mu(\dx x)\Big|\leq|\gamma_{\widetilde{U}}|\epsilon.
\end{align*}
Since the above statement holds for every $\epsilon>0$, we can conclude that the limit is zero. For the contribution of the small jumps of $\eta$, recall that $S_{\eta}\ast\mu$ is integrable with respect to~$\lambda$ by Lemma~\ref{lem-functions}. Therefore, we find that
\begin{displaymath}
0\leq\lim_{t\rightarrow\infty}\frac{1}{\ln(t)}\int_1^t\frac{1}{z^2}(S_{\eta}\ast\mu)(z)\dx z\leq\lim_{t\rightarrow\infty}\frac{1}{\ln(t)}\int_{\RR}(S_{\eta}\ast\mu)(z)\dx z=0.
\end{displaymath}
For the summand involving $S_{\widetilde{U}}$, splitting up the inner integral leads to
\begin{align*}
\int_1^t\frac{1}{z^2}\int_0^{\infty}xS_{\widetilde{U}}(\tfrac{z}{x})\mu(\dx x)\dx z&=\int_1^t\frac{1}{z^2}\int_{\frac{z}{2}}^{2z}xS_{\widetilde{U}}(\tfrac{z}{x})\mu(\dx x)\dx z+\int_1^t\frac{1}{z^2}\int_{2z}^{\infty}xS_{\widetilde{U}}(\tfrac{z}{x})\mu(\dx x)\dx z,
\end{align*}
as $x<\frac{z}{2}$ implies that $\frac{z}{x}-1>1$ and, therefore, we have $S_{\widetilde{U}}(\frac{z}{x})=0$ in this case. Since~$S_{\widetilde{U}}(\frac{z}{x})$ is nonnegative by~\eqref{def-S-U}, a direct calculation leads to
\begin{align*}
\int_{2z}^{\infty}xS_{\widetilde{U}}(\tfrac{z}{x})\mu(\dx x)&=\int_{2z}^{\infty}\int_{-1}^{\frac{z}{x}-1}\big(z-x(y+1)\big)\nu_{\widetilde{U}}(\dx y)\mu(\dx x)\\
&=\int_{-1}^{-\frac{1}{2}}\int_{2z}^{\frac{z}{y+1}}\big(z-x(y+1)\big)\mu(\dx x)\nu_{\widetilde{U}}(\dx y)\\
&\leq\int_{-1}^{-\frac{1}{2}}\big(z-2z(y+1)\big)\int_{2z}^{\frac{z}{y+1}}\mu(\dx x)\nu_{\widetilde{U}}(\dx y)\\
&\leq z\mu([2z,\infty))\nu_{\widetilde{U}}([-1,-\tfrac{1}{2}]),
\end{align*}
which, since $\mu([2z,\infty))\rightarrow0$ as $z\rightarrow\infty$, implies that the term is in $o(z)$. We can thus apply the same reasoning as for the contribution of the drift terms and conclude that
\begin{align*}
\lim_{t\rightarrow\infty}\frac{1}{\ln(t)}\int_1^t\frac{1}{z^2}\int_{2z}^{\infty}xS_{\widetilde{U}}(\tfrac{z}{x})\mu(\dx x)\dx z=0.
\end{align*}
If $z/2\leq x\leq2z$, consider
\begin{align*}
\int_1^t\int_{\frac{z}{2}}^{2z}\frac{x}{z^2}S_{\widetilde{U}}(\tfrac{z}{x})\mu(\dx x)\dx z&=\int_{\frac{1}{2}}^{2t}\int_{\max\{\frac{x}{2},1\}}^{\min\{t,2x\}}\frac{x}{z^2}S_{\widetilde{U}}(\tfrac{z}{x})\dx z\mu(\dx x)\\
&\leq\int_{\frac{1}{2}}^{2t}\int_{\frac{x}{2}}^{x(1-\varepsilon)}\frac{x}{z^2}S_{\widetilde{U}}(\tfrac{z}{x})\dx z\mu(\dx x)+\int_{\frac{1}{2}}^{2t}\int_{x(1-\varepsilon)}^{x(1+\varepsilon)}\frac{x}{z^2}S_{\widetilde{U}}(\tfrac{z}{x})\dx z\mu(\dx x)\\
&\quad+\int_{\frac{1}{2}}^{2t}\int_{x(1+\varepsilon)}^{2x}\frac{x}{z^2}S_{\widetilde{U}}(\tfrac{z}{x})\dx z\mu(\dx x)
\end{align*}
for some $\varepsilon\in(0,1/2)$. In the case that the values of $z$ are bounded away from $z=x$, we can use~\eqref{def-S-U} to estimate $S_{\widetilde{U}}(\frac{z}{x})$ by a constant $C_{\varepsilon}>0$, implying that
\begin{align*}
\int_{x(1+\varepsilon)}^{2x}\frac{1}{z^2}S_{\widetilde{U}}(\tfrac{z}{x})\dx z\leq C_{\varepsilon}\int_{x(1+\varepsilon)}^{2x}\frac{1}{z^2}\dx z=C_{\varepsilon}\Bigl(\frac{1}{x(1+\varepsilon)}-\frac{1}{2x}\Bigr)
\end{align*}
with a similar estimate also holding for $z\in[x/2,x(1-\varepsilon)]$. If the values of $z$ are close to the singularity at $z=x$, we find that
\begin{align}
\int_{x(1-\varepsilon)}^{x(1+\varepsilon)}\frac{1}{z^2}S_{\widetilde{U}}(\tfrac{z}{x})\dx z&\leq\frac{1}{x^2(1-\varepsilon)^2}\int_{x(1-\varepsilon)}^xS_{\widetilde{U}}(\tfrac{z}{x})\dx z+\frac{1}{x^2}\int_x^{x(1+\varepsilon)}S_{\widetilde{U}}(\tfrac{z}{x})\dx z\nonumber\\
&=\frac{1}{x(1-\varepsilon)^2}\int_{-\epsilon}^0S_{\widetilde{U}}(t+1)\dx t+\frac{1}{x}\int_0^{\varepsilon}S_{\widetilde{U}}(t+1)\dx t\label{eq-limitest1}
\end{align}
by a suitable substitution. Note that both integrals are finite by the definition of $\nu_{\widetilde{U}}$, since
\begin{align}
\int_0^{\varepsilon}S_{\widetilde{U}}(t+1)\dx t&=\int_0^{\varepsilon}\int_t^1(y-t)\nu_{\widetilde{U}}(\dx y)\dx t\nonumber\\
&=\int_0^1\int_0^{\min\{\varepsilon,y\}}(y-t)\dx t\nu_{\widetilde{U}}(\dx y)\leq\int_0^1y^2\nu_{\widetilde{U}}(\dx y),\label{eq-limitest2}
\end{align}
and a similar estimate holds for $t\in[-\varepsilon,0)$. Therefore, one obtains an estimate in terms of $\frac{1}{x}$ in all cases, implying
\begin{align*}
0&\leq\lim_{t\rightarrow\infty}\frac{1}{\ln(t)}\int_1^t\frac{1}{z^2}\int_{\frac{z}{2}}^{2z}xS_{\widetilde{U}}(\tfrac{z}{x})\mu(\dx x)\dx z\leq\lim_{t\rightarrow\infty}\frac{1}{\ln(t)}\int_{\frac{1}{2}}^tx\frac{\widetilde{C}}{x}\mu(\dx x)\\
&=\widetilde{C}\lim_{t\rightarrow\infty}\frac{1}{\ln(t)}\mu([\tfrac{1}{2},t])=0
\end{align*}
with a suitable constant $\widetilde{C}$. For the term corresponding to the big jumps of $\eta$, observe that the function~$B_{\eta}$ is bounded and satisfies $\lim_{|z|\rightarrow\infty}B_{\eta}(z)=0$ by~\eqref{def-B-eta}. This implies that also~${(B_{\eta}\ast\mu)(z)\rightarrow0}$, as can be seen by partitioning the domain of integration of the convolution with respect to the values of the function $B_{\eta}$. Thus, we also find that
\begin{align*}
\lim_{|z|\rightarrow\infty}\frac{1}{z}\int_0^z(B_{\eta}\ast\mu)(x)\dx x=0,
\end{align*}
i.e. the corresponding summand in $G$ is in $o(z)$, which, in combination with the above arguments is enough to conclude that this term also does not contribute to the limit in~\eqref{limit-constants1}. Similarly, we observe that also $B_{\widetilde{U}}$ is bounded and satisfies $\lim_{|t|\rightarrow\infty}B_{\widetilde{U}}(t)=0$ as~$\nu_{\widetilde{U}}((1,\infty))<\infty$, which, together with~$\mu$ being a finite measure implies that
\begin{align*}
\lim_{z\rightarrow\infty}\int_0^zB_{\widetilde{U}}(\tfrac{z}{x})\mu(\dx x)=0
\end{align*}
by dominated convergence. Therefore, the corresponding antiderivative appearing in $G$ is in $o(z)$ as desired.
\end{proof}
%
%%% Proof of Functional Equation
%
Using the above lemmas, we can now prove Theorem \ref{main}.
\begin{proof}[Proof of Theorem \ref{main}]
Starting from \eqref{eq-invmeasure}, we first rewrite the left-hand side according to Lemmas \ref{rewriting1} and \ref{rewriting2} to arrive at
\begin{displaymath}
\int_{\RR}A^{\widetilde{V}}f(x)\mu(\dx x)=\int_{\RR}f''(z)G(\dx z)=0.
\end{displaymath}
Recall that this equation holds for every $f\in C_c^{\infty}(\RR)$ by Corollary~\ref{cor-testfunctions}. Using the results from Lemma \ref{odesolve}, we find that $G''$ equals the zero distribution and thus~${G=C_1z\dx z+C_2\dx z}$ for some constants $C_1,C_2\in\RR$. Identifying the equivalent expressions for $G$ from Lemma~\ref{rewriting2} and~\ref{odesolve} now yields
\begin{align}
C_1z\dx z+C_2\dx z=&\Bigl(\frac{1}{2}\sigma_{\eta}^2+\frac{1}{2}z^2\sigma_{\widetilde{U}}^2\Bigr)\mu(\dx z)+(S_{\eta}\ast\mu)(z)\dx z\nonumber\\
&\quad+\Big(\I_{\{z>0\}}\int_0^{\infty}xS_{\widetilde{U}}(\tfrac{z}{x})\mu(\dx x)+\I_{\{z<0\}}\int_{-\infty}^0|x|S_{\widetilde{U}}(\tfrac{z}{x})\mu(\dx x)\Big)\dx z\nonumber\\
&\quad-\int_{0+}^z\bigl(\gamma_{\eta}+x\gamma_{\widetilde{U}}\bigr)\mu(\dx x)\dx z-\int_{0+}^z(B_{\eta}\ast\mu)(x)\dx x\dx z\nonumber\\
&\quad-\int_{0+}^z\int_{0+}^tB_{\widetilde{U}}(\tfrac{t}{x})\mu(\dx x)\dx t\dx z.\label{equation-G}
\end{align}
In order to arrive at~\eqref{eq-mu}, the values of $C_1$ and $C_2$ have to be identified. To determine~$C_1$, observe that
\begin{equation*}
\frac{1}{\ln(t)}\int_1^t\frac{1}{z^2}(C_1z+C_2)\dx z=C_1+C_2\frac{1-\frac{1}{t}}{\ln(t)}\ \rightarrow C_1,\ t\rightarrow\infty,
\end{equation*}
such that we can give its value by applying the above transformations to both sides of~\eqref{equation-G} and letting $t\rightarrow\infty$. From Lemma~\ref{constants}, this limit is equal to zero. Renaming $K=C_2$, we arrive at~\eqref{eq-mu}.
\end{proof}

\begin{proof}[Proof of Corollary~\ref{existence-density}]
(i) and (ii), Existence: From the form of~\eqref{eq-mu}, we see that
\begin{displaymath}
\Big(\frac{1}{2}\sigma_{\eta}^2+\frac{1}{2}z^2\sigma_{\widetilde{U}}^2\Big)\mu(\dx z)=H(z)\dx z
\end{displaymath}
for some locally integrable function $H$. It follows that $\mu$ has a density $f_{\mu}$ on $\RR$ whenever $\sigma_{\eta}^2>0$ and that $\mu|_{\RR\setminus\{0\}}$ has a density $f_{\mu}$ on $\RR\setminus\{0\}$ whenever $\sigma_{\widetilde{U}}^2>0=\sigma_{\eta}^2$. In both cases, Corollary~\ref{corollary-equation-density} implies that $f_{\mu}$ must satisfy~\eqref{equation-density} for~$\lambda$-a.e. $z\in\RR$.

(iii) Since $S_{\eta}\geq0$ and $S_{\widetilde{U}}\geq0$, the term $\smash{(\frac{1}{2}\sigma_{\eta}^2+\frac{1}{2}z^2\sigma_{\widetilde{U}}^2)f_{\mu}}$ can be bounded by the right-hand side of~\eqref{equation-density}. Observe that all quantities in this bound, apart from $\I_{\{z<0\}}\gamma_{\eta}\mu(\{0\})$ if $\gamma_{\eta}\mu(\{0\})\neq0$, are continuous functions in $z$. In particular, the right-hand side of~\eqref{equation-density} is locally bounded in $z\in\RR$, continuous on $\RR\setminus\{0\}$ and, whenever $\mu(\{0\})=0$ (which is in particular satisfied if $\sigma_{\eta}^2>0$), also continuous on $\RR$. Observing further that $B_{\eta}\ast f_{\mu}$ is integrable and $B_{\widetilde{U}}$ is bounded by definition (cf.~Lemma~\ref{lem-functions}), we see that the right-hand side of~\eqref{equation-density} can be bounded by $M_1+M_2|z|$ for $z\in\RR$ and suitable constants~$M_1,M_2\geq0$, yielding the desired bound
\begin{displaymath}
\Big(\frac{1}{2}\sigma_{\eta}^2+\frac{1}{2}z^2\sigma_{\widetilde{U}}^2\Big)f_{\mu}(z)\leq M_1+M_2|z|,\quad \forall z\neq0.
\end{displaymath}

(ii), Continuity: Let $\sigma_{\widetilde{U}}^2>0$ and $\sigma_{\eta}^2\geq0$. Since the right-hand side of~\eqref{equation-density} is continuous on $\RR\setminus\{0\}$, it suffices to show that $S_{\eta}$, $S_{\eta}\ast f_{\mu}$, as well as the mappings $\smash{z\mapsto\I_{\{z>0\}}\int_0^{\infty} S_{\widetilde{U}}(\frac{z}{x})f_{\mu}(x)\dx x}$ and $\smash{z\mapsto\I_{\{z<0\}}\int_{-\infty}^0|x|S_{\widetilde{U}}(\frac{z}{x})f_{\mu}(x)\dx x}$ are continuous on \linebreak$\RR\setminus\{0\}$. Write
\begin{displaymath}
S_{\eta}(z)=\int_{\varepsilon}^{\infty}(y-z)\I_{[z,\infty)}(y)\nu_{\eta}|_{[-1,1]}(\dx y)
\end{displaymath}
for $z>\varepsilon>0$, and observe that the function $\smash{z\mapsto(y-z)\I_{[z,\infty)}(y)}$ is continuous in $z_0>\varepsilon$ for all values of $y$. Thus, an application of Lebesgue's dominated convergence theorem yields that $S_{\eta}$ is continuous in $z_0>\varepsilon$. Since $\varepsilon>0$ was arbitrary and we can apply a similar argument for $z_0<0$, it follows that $S_{\eta}$ is continuous on $\RR\setminus\{0\}$. To show that~$S_{\eta}\ast f_{\mu}$ is continuous in $z_0>0$, let $\varepsilon\in(0,z_0)$ as well as $\delta\in(0,1)$ and decompose
\begin{equation}\label{decomp-S-eta}
S_{\eta}(z)=S_{\eta}^{\delta,1}+S_{\eta}^{\delta,2},
\end{equation}
where the functions on the right-hand side are defined similar to~\eqref{def-S-eta} with $\nu_{\eta}|_{[-1,1]}$ replaced by $\nu_{\eta}|_{[-\delta,\delta]}$ or $\nu_{\eta}|_{[-1,1]\setminus[-\delta,\delta]}$ for $\smash{S_{\eta}^{\delta,1}}$ or $S_{\eta}^{\delta,2}$, respectively. Then $S_{\eta}^{\delta,1}$ and $S_{\eta}^{\delta,2}$ are continuous on $\RR\setminus\{0\}$ and $S_{\eta}^{\delta,2}$ is bounded by $\nu_{\eta}([-1,1]\setminus[-\delta,\delta])<\infty$ by definition. The latter implies that $\smash{S_{\eta}^{\delta,2}\ast f_{\mu}}$ is continuous on $\RR$ for every $\delta\in(0,1)$ (see e.g.~\cite[Thm.~14.8]{Schilling2011}). For the treatment of $\smash{S_{\eta}^{\delta,1}\ast f_{\mu}}$, recall from Lemma~\ref{lem-functions} that $S_{\eta}$ is integrable with respect to $\lambda$. Since $\smash{S_{\eta}^{\delta,1}}$ converges (point-wise) to zero as $\delta\downarrow0$ and $\smash{S_{\eta}^{\delta,1}}\leq S_{\eta}$, it follows that
\begin{displaymath}
\lim_{\delta\downarrow0}\int_{-\frac{z_0}{4}}^{\frac{z_0}{4}}S_{\eta}^{\delta,1}(z)\dx z=0
\end{displaymath}
by dominated convergence. By part (iii) of the Corollary, we can bound $f_{\mu}$ by a constant~$M_3>0$ on $[z_0/4,7z_0/4]$. For $z\in(z_0/2,3z_0/2)$ and $0<\delta<z_0/4$ we have $\smash{S_{\eta}^{\delta,1}}=0$ for $|y|>z_0/4$ and hence
\begin{displaymath}
\big(S_{\eta}^{\delta,1}\ast f_{\mu}\big)(z)=\int_{\frac{z_0}{4}}^{\frac{z_0}{4}}f_{\mu}(z-x)S_{\eta}^{\delta,1}(x)\dx x\leq M_3\int_{-\frac{z_0}{4}}^{\frac{z_0}{4}}S_{\eta}^{\delta,1}(x)\dx x.
\end{displaymath}
Choosing $\delta$ small enough, the above estimate on the right-hand side becomes arbitrarily small. Together with the previously established continuity of $\smash{S_{\eta}^{\delta,2}}$ and~\eqref{decomp-S-eta}, this shows continuity of $S_{\eta}\ast f_{\mu}$ in $z_0>0$. Applying a similar argument for $z_0<0$, we can conclude that $S_{\eta}\ast f_{\mu}$ is continuous on $\RR\setminus\{0\}$.

It remains to consider the terms involving $S_{\widetilde{U}}$. First, we establish continuity of the mapping $z\mapsto\smash{\int_0^{\infty}xS_{\widetilde{U}}(\tfrac{z}{x})f_{\mu}(x)\dx x}$ in $z_0>0$. As for~\eqref{decomp-S-eta}, let $\delta\in(0,1)$ and decompose
\begin{equation}\label{decomp-S-U}
S_{\widetilde{U}}(z)=S_{\widetilde{U}}^{\delta,1}+S_{\widetilde{U}}^{\delta,2},
\end{equation}
where the quantities $\smash{S_{\widetilde{U}}^{\delta,1}}$ and $\smash{S_{\widetilde{U}}^{\delta,2}}$ are defined similar to~\eqref{def-S-U} with $\smash{\nu_{\widetilde{U}}|_{[-1,1]}}$ replaced by $\smash{\nu_{\widetilde{U}}|_{[-\delta,\delta]}}$ or $\smash{\nu_{\widetilde{U}}|_{[-1,1]\setminus[-\delta,\delta]}}$, respectively. As in the treatment of $S_{\eta}$, a bound for $\smash{S_{\widetilde{U}}^{\delta,2}}$ is readily obtained from the definition since
\begin{displaymath}
S_{\widetilde{U}}^{\delta,2}(z)\leq\int_{-1}^{z-1}(z-1-(-1))\nu_{\widetilde{U}}|_{[-1,1]\setminus[-\delta,\delta]}(\dx y)\leq z\nu_{\widetilde{U}}([-1,1]\setminus[-\delta,\delta])
\end{displaymath}
for $z\in[0,1]$ and, setting $\smash{M_4^{\delta}}=\nu_{\widetilde{U}}([-1,1]\setminus[-\delta,\delta])$,
\begin{displaymath}
S_{\widetilde{U}}^{\delta,2}(z)\leq\int_0^1y\nu_{\widetilde{U}}|_{[-1,1]\setminus[-\delta,\delta]}(\dx y)\leq M_4^{\delta}
\end{displaymath}
for $z>1$. Further, $\smash{S_{\widetilde{U}}^{\delta,2}}$ is continuous on $[0,\infty)\setminus\{1\}$, as can be seen from applying a similar argument as for $\smash{S_{\eta}^{\delta,2}}$. Writing
\begin{displaymath}
\int_0^{\infty}xS_{\widetilde{U}}^{\delta,2}(\tfrac{z}{x})f_{\mu}(x)\dx x=\int_0^{\infty}x\I_{(0,z)}(x)S_{\widetilde{U}}^{\delta,2}(\tfrac{z}{x})f_{\mu}(x)\dx x+\int_0^{\infty}x\I_{(z,\infty)}(x)S_{\widetilde{U}}^{\delta,2}(\tfrac{z}{x})f_{\mu}(x)\dx x,
\end{displaymath}
the integrand can be bounded by $\smash{(x\I_{(0,z)}(x)M_4^{\delta}+z\I_{(z,\infty)}(x)M_4^{\delta})f_{\mu}(x)}$ such that the continuity of the mapping ${z\mapsto\smash{\int_0^{\infty}xS_{\widetilde{U}}^{\delta,2}(\tfrac{z}{x})f_{\mu}(x)\dx x}}$ in $z_0>0$ follows by dominated convergence. Since we can apply a similar argument for the continuity in $z_0<0$ of the corresponding function on the negative real numbers and we have for $z_0=0$ that
\begin{displaymath}
\lim_{z\downarrow0}\int_0^{\infty}xS_{\widetilde{U}}^{\delta,2}(\tfrac{z}{x})f_{\mu}(x)\dx x=\int_0^{\infty}xS_{\widetilde{U}}^{\delta,2}(0)f_{\mu}(x)\dx x=0,
\end{displaymath}
it follows that the mapping
\begin{displaymath}
z\mapsto\I_{\{z>0\}}\int_0^{\infty}xS_{\widetilde{U}}^{\delta,2}(\tfrac{z}{x})f_{\mu}(x)\dx x+\I_{\{z<0\}}\int_{-\infty}^0|x|S_{\widetilde{U}}^{\delta,2}(\tfrac{z}{x})f_{\mu}(x)\dx x
\end{displaymath}
is continuous on $\RR$. Therefore, it only remains to consider the term involving~$\smash{S_{\widetilde{U}}^{\delta,1}}$. In order to do so, observe that the support of $\smash{S_{\widetilde{U}}^{\delta,1}}$ is contained in the interval $[1-\delta,1+\delta]$, that $\smash{S_{\widetilde{U}}^{\delta,1}}\leq S_{\widetilde{U}}$ by definition and that, as a consequence of the integrability of $S_{\widetilde{U}}$ (cf. Lemma~\ref{lem-functions}) we have that
\begin{displaymath}
0\leq\lim_{\delta\downarrow0}\int_{\RR}S_{\widetilde{U}}^{\delta,1}(x)\dx x\leq\lim_{\delta\downarrow0}\int_{1-\delta}^{1+\delta}S_{\widetilde{U}}(x)\dx x=0.
\end{displaymath}
Using the substitution $v=z/x$ for $z>0$, it follows that
\begin{displaymath}
\int_0^{\infty}xS_{\widetilde{U}}^{\delta,1}(\tfrac{z}{x})f_{\mu}(x)\dx x=\int_{\frac{z}{1+\delta}}^{\frac{z}{1-\delta}}xS_{\widetilde{U}}^{\delta,1}(\tfrac{z}{x})f_{\mu}(x)\dx x=\int_{1-\delta}^{1+\delta}\frac{z^2}{v^3}S_{\widetilde{U}}^{\delta,1}(v)f_{\mu}(\tfrac{z}{v})\dx v.
\end{displaymath}
Since $f_{\mu}$ is locally bounded on $\RR\setminus\{0\}$ by part (iii) of the Corollary, the above quantity becomes arbitrarily small for sufficiently small $\delta>0$ when $z\in(z_0/2,3z_0/2)$ for $z_0>0$. Together with~\eqref{decomp-S-U} and the already established continuity of the terms involving $\smash{S_{\eta}^{\delta,2}}$, it follows that the mapping
\begin{displaymath}
z\mapsto\I_{\{z>0\}}\int_0^{\infty}xS_{\widetilde{U}}(\tfrac{z}{x})f_{\mu}(x)\dx x+\I_{\{z<0\}}\int_{-\infty}^0|x|S_{\widetilde{U}}(\tfrac{z}{x})f_{\mu}(x)\dx x
\end{displaymath}
is continuous on $\RR\setminus\{0\}$. The desired continuity of $f_{\mu}$ on $\RR\setminus\{0\}$ hence follows from~\eqref{equation-density}.

(i), Continuity: Now assume that $\sigma_{\eta}^2>0$. As $\mu$ has a density on $\RR$, it follows that $\mu(\{0\})=0$ and using the same argument as in the proof of part (ii), it is sufficient to show that the mappings~$z\mapsto (S_{\eta}\ast f_{\mu})(z)$ and
\begin{displaymath}
z\mapsto\I_{\{z>0\}}\int_0^{\infty}xS_{\widetilde{U}}(\tfrac{z}{x})f_{\mu}(x)\dx x+\I_{\{z<0\}}\int_{-\infty}^0|x|S_{\widetilde{U}}(\tfrac{z}{x})f_{\mu}(x)\dx x
\end{displaymath}
are continuous in $z=0$. Observe that by (iii), $f_{\mu}$ is not only locally bounded on $\RR\setminus\{0\}$, but on $\RR$ whenever $\sigma_{\eta}^2>0$, such that we can use the methods from part (ii) also for $z=0$ in this case. Note that $f_{\mu}$ is in particular bounded on $[-1,1]$, and that
\begin{displaymath}
\lim_{\delta\downarrow0} \int_{-1}^1 S_{\eta}^{\delta,1}(z)\dx z=\lim_{\delta\downarrow0} \int_{-1}^1 S_{\widetilde{U}}^{\delta,1}(z)\dx z=0.
\end{displaymath}
The terms including $S_{\widetilde{U}}^{\delta,1}$ and $S_{\eta}^{\delta,1}$ thus become arbitrarily small in a neighborhood of zero. Since the terms involving $\smash{S_{\widetilde{U}}^{\delta,2}}$ and $\smash{S_{\eta}^{\delta,2}}$ are again continuous, we find that $f_{\mu}$ is also continuous in $z=0$. This finishes the proof.
\end{proof}

\begin{remark}\rm{
It seems tempting to iterate the proof of Corollary~\ref{existence-density} to obtain further smoothness properties of $f_{\mu}$. Such an argument would require being able to show that $f_{\mu}\in C(\RR)$ implies $S_{\eta}\ast f_{\mu}\in C^1(\RR)$ or at least $S_{\eta}\ast f_{\mu}\in C^1(\RR\setminus\{0\})$, as well as similar statements for the other quantities on the right-hand side of~\eqref{equation-density}. This claim is, however, not true in general. A counterexample is given by $\nu_{\eta}(\dx x)=x^{-5/2}\I_{(0,1)}(x)\dx x$ and $f\in C_c(\RR)$ being a density that satisfies $f(x)=c((x-2)^{1/3}+2)$ for $x\in[2,3]$ and $f(x)=(2-(2-x)^{1/3})$ for $x\in[1,2]$, where $c>0$ is a suitable norming constant. Since $S_{\eta}(x)\sim4/3x^{-1/2}$ as $x\downarrow0$ and $f'(x)=\tfrac{c}{3}|x-2|^{-2/3}$ for $x\in[1,3]$, an application of Fatou's lemma shows that
\begin{displaymath}
\liminf_{x\downarrow 2}\frac{(S_{\eta}\ast f)(x)-(S_{\eta}\ast f)(2)}{x-2}=\infty
\end{displaymath}
such that $S_{\eta}\ast f$ is not differentiable in $x=2$. Hence, an easy iterative argument seems not to be possible in the general case considered in Theorem~\ref{main} and Corollary~\ref{existence-density}. Observe, however, that Corollary~\ref{corollary-diffdensity} gives conditions for $\smash{f_{\mu}\in C^1(\RR\setminus\{0\})}$ and Example~\ref{ex-potentialmeasure} below gives a concrete example when~$\smash{\sigma_{\eta}^2>0}$ and $\smash{f_{\mu}\in C^1(\RR\setminus\{0\})\setminus C^1(\RR)}$. Furthermore, note that restricting the characteristics of the L\'evy processes involved may yield much stronger smoothness properties than the general case, e.g. if $\eta_t=t$, where the density of the law of the killed exponential functional with $q\geq0$ is infinitely often differentiable on $\RR\setminus\{0\}$ for most choices of $\xi$~(see \cite[Thm.~2.4(3)]{PatieSavov}).
}\end{remark}

\begin{proof}[Proof of Corollary~\ref{corollary-finitemoments}]
(i) Observe first that $\smash{S_{\eta}^b=S_{\eta}^{FM}-S_{\eta}}$ and $\smash{S_{\widetilde{U}}^b=S_{\widetilde{U}}^{FM}-S_{\widetilde{U}}}$ are bounded functions vanishing at infinity and that $\smash{S_{\widetilde{U}}^b=0}$ for $z\leq1$. Thus, the convolution~$\smash{S_{\eta}^{FM}\ast\mu}$ can be written as a sum of an integrable and a bounded function and  hence is locally integrable with respect to~$\lambda$. The analogue of Lemma~\ref{lem-functions2} involves the measure~$\rho^{FM}$, which takes the form
\begin{displaymath}
\rho^{FM}(\dx z)=\rho(\dx z)+\Big(\I_{\{z>0\}}\int_0^zS_{\widetilde{U}}^b(\tfrac{z}{x})\mu(\dx x)-\I_{\{z<0\}}\int_z^0xS_{\widetilde{U}}(\tfrac{z}{x})\mu(\dx x)\Big)\dx z,
\end{displaymath}
from which it is visible that the right-hand side of~\eqref{eq-mu-fm} defines a locally finite measure. As the moment condition implies that both $\int_{\RR\setminus[-1,1]}|x|\nu_{\eta}(\dx y)$ and $\int_{(1,\infty)}|x|\nu_{\widetilde{U}}(\dx y)$ are finite, it follows that
\begin{align}
&\int_{\RR}\big(f(x+y)-f(x)-yf'(x)\I_{|y|\leq1}(y)\big)\nu_{\eta}(\dx y)+\gamma_{\eta}f'(x)\nonumber\\
&\quad=\int_{\RR}\big(f(x+y)-f(x)-yf'(x)\big)\nu_{\eta}(\dx y)+\gamma_{\eta}^1f'(x)\label{integral-finitemoments}
\end{align}
with a similar relation also holding true for $\smash{\widetilde{U}}$. One now follows the proofs of Lemmas~\ref{rewriting1} and~\ref{rewriting2}, i.e. considers the integral with respect to $\mu$, shows that Fubini's Theorem is applicable for the terms involving multiple integrals and thus recovers a similar distribution~$G^{FM}(\dx z)$. To show e.g. that
\begin{displaymath}
\int_{1+}^{\infty}\int_0^{\infty}\int_x^{x+xy}|f''(t)(x+xy-t)|\dx t\mu(\dx x)\nu_{\widetilde{U}}(\dx y)<\infty
\end{displaymath}
for $f\in C_c^{\infty}(\RR)$ with $\supp(f)\subset[-R,R]$, observe that the inner integral can be estimated by $RC(x+xy)\leq R^2C(1+y)$ for a suitable constant $C\geq0$ such that $|f''(t)|\leq C$ for~$t\in[-R,R]$. Together with the moment condition, this yields that the triple integral, too, is finite. Following the remaining steps of the proof of Theorem~\ref{main}, we also obtain that
\begin{displaymath}
\lim_{t\rightarrow\infty}\frac{1}{\ln(t)}\int_1^z\frac{1}{z^2}\Big(S_{\eta}^b(z)+\int_0^zxS_{\widetilde{U}}^b(\tfrac{z}{x})\mu(\dx x)\Big)\dx z=0
\end{displaymath}
which yields Equation~\eqref{eq-mu-fm} as claimed.\\
(ii) Since $\smash{\ew|\cE(U)_1|=\ew\cE(U)_1=\re^{\ew U_1}}$ (see~\cite[Prop.~3.1]{Behme2011}) and $\smash{\gamma_{\widetilde{U}}^1=\ew \widetilde{U}_1=\ew U_1-q}$ by definition, the condition $\ew|\cE(U)_1|<\re^q$ is equivalent to $\gamma_{\widetilde{U}}^1<0$. Further, $\ew|\cE(U)_1|<\re^q$ implies $\smash{\int_{\RR}|x|\mu(\dx x)<\infty}$, as is shown in~\cite[Thm.~3.1]{Behme2011}. Let $\smash{G^{FM}(\dx z)}$ denote the right-hand side of~\eqref{eq-mu-fm}. To determine the value of the constant, we use a similar approach as in Lemma~\ref{constants}, showing that
\begin{equation}\label{eq-limitconstants2}
\lim_{t\rightarrow\infty}t\int_t^{\infty}\frac{1}{z^2}G^{FM}(\dx z)=-\int_{0+}^{\infty}\big(\gamma_{\eta}^1+x\gamma_{\widetilde{U}}^1\big)\mu(\dx x).
\end{equation}
To see that~\eqref{eq-limitconstants2} holds, observe first that $\lim_{t\rightarrow\infty}t\mu((t,\infty))=0$ as a consequence of $t\mu((t,\infty))\leq\int_{(t,\infty)}|x|\mu(\dx x)<\infty$. This implies that
\begin{displaymath}
\lim_{t\rightarrow\infty}t\int_t^{\infty}\Big(\frac{\sigma_{\eta}^2}{2z^2}+\frac{\sigma_{\widetilde{U}}^2}{2}\Big)\mu(\dx z)=0.
\end{displaymath}
Further, we find that
\begin{displaymath}
\lim_{t\rightarrow\infty}t\int_t^{\infty}\frac{1}{z^2}\big(S_{\eta}^{FM}\ast\mu\big)(z)\dx z\leq\lim_{t\rightarrow\infty}\int_t^{\infty}(S_{\eta}^{FM}\ast\mu)(z)\dx z=0,
\end{displaymath}
since $S_{\eta}^{FM}\ast\mu=S_{\eta}\ast\mu+S_{\eta}^b\ast\mu$, where $S_{\eta}\ast\mu$ is integrable with respect to $\lambda$ and $S_{\eta}^b\ast\mu$ is bounded with~$\lim_{t\rightarrow\infty} S_{\eta}^b\ast\mu(z)=0$. Next, observe that $S_{\widetilde{U}}$ is bounded on~${[0,\infty)\setminus[1/2,2]}$ and that $\lim_{z\rightarrow\infty} S_{\widetilde{U}}(z)=0$. Since $\int |x|\mu(\dx x)<\infty$, an application of Lebesgue's dominated convergence theorem yields that
\begin{displaymath}
\lim_{z\rightarrow\infty}\int_{(0,\frac{z}{2})\cup(2z,\infty)}xS_{\widetilde{U}}^{FM}(\tfrac{z}{x})\mu(\dx x)=0
\end{displaymath}
and hence
\begin{displaymath}
\lim_{t\rightarrow\infty}t\int_t^{\infty}\frac{1}{z^2}\int_{(0,\frac{z}{2})\cup(2z,\infty)}xS_{\widetilde{U}}^{FM}(\tfrac{z}{x})\mu(\dx x)\dx z=0.
\end{displaymath}
For $z/2\leq x\leq z$ we find, similar to~\eqref{eq-limitest1} and~\eqref{eq-limitest2}, that for some constant $C>0$
\begin{align*}
t\int_t^{\infty}\frac{1}{z^2}\int_{\frac{z}{2}}^{2z}xS_{\widetilde{U}}^{FM}(\tfrac{z}{x})\mu(\dx x)
\leq t\int_{\frac{t}{2}}^{\infty}x\int_{\frac{x}{2}}^{\infty}\frac{1}{z^2}S_{\widetilde{U}}^{FM}(\tfrac{z}{x})\dx z\mu(\dx x)
\leq t\int_{\frac{t}{2}}^{\infty}x\frac{C}{x}\mu(\dx x),
\end{align*}
where the right-hand side converges to zero as $t\rightarrow\infty$. Lastly, observe that
\begin{displaymath}
\lim_{t\rightarrow\infty}t\int_t^{\infty}\frac{1}{z^2}\int_{0+}^z\big(\gamma_{\eta}^1+x\gamma_{\widetilde{U}}^1\big)\mu(\dx x)\dx z=\int_{0+}^{\infty}\big(\gamma_{\eta}^1+x\gamma_{\widetilde{U}}^1\big)\mu(\dx x),
\end{displaymath}
which yields the value of $K$ and thus finishes the proof of~\eqref{eq-limitconstants2}. That the constant can also be written as $\smash{\int_{-\infty}^0\big(\gamma_{\eta}^1+x\gamma_{\widetilde{U}}^1\big)\mu(\dx x)}$ follows by a similar argument considering~$|t|\smash{\int_{-\infty}^tz^{-2}G^{FM}(\dx z)}$ for ${t\rightarrow-\infty}$, or alternatively from ${\ew(V_{q,\xi,\eta})\ew(\widetilde{U}_1)=-\ew(\eta_1)}$ (cf.~\cite[Thm.~3.3a]{Behme2011}). Now assume that $\sigma_{\eta}^2+\sigma_{\widetilde{U}}^2>0$. By Corollary~\ref{existence-density}, $\mu|_{\RR\setminus\{0\}}$ has a density $f_{\mu}$. Note, however, that rearranging the terms in~\eqref{eq-mu-fm} and using the positivity of~$S_{\eta}^{FM}$ and~$S_{\widetilde{U}}^{FM}$ as in the proof of Corollary~\ref{existence-density} leads to
\begin{align*}
\Bigl(\frac{1}{2}\sigma_{\eta}^2+\frac{1}{2}z^2\sigma_{\widetilde{U}}^2\Bigr)f_{\mu}(z)\leq K+\int_{0+}^z\big(\gamma_{\eta}^1+\gamma_{\widetilde{U}}^1\big)\mu(\dx x)\leq |K|+|\gamma_{\eta}^1|+|\gamma_{\widetilde{U}}^1|\int_{\RR}|x|\mu(\dx x)<\infty
\end{align*}
here due to the moment condition. Thus, $f_{\mu}$ is bounded.
\end{proof}

\begin{proof}[Proof of Corollary~\ref{corollary-finitevariation}]
Observe first that $\smash{B_{\eta}^I=B_{\eta}^{FV}-B_{\eta}}$ and $\smash{B_{\widetilde{U}}^I=B_{\widetilde{U}}^{FV}-B_{\widetilde{U}}}$ are integrable with respect to $\lambda$ due to the finite variation condition. In particular, we have that~$B_{\eta}^{FV}\ast\mu=B_{\eta}\ast\mu+B_{\eta}^I\ast\mu$ is the sum of a bounded and an integrable function and hence locally integrable. In order to obtain the analogue of Lemma~\ref{lem-functions} in the finite variation case, one needs to show that the mapping $z\mapsto\smash{\int_{0+}^{\infty}B_{\widetilde{U}}(\tfrac{z}{x})\mu(\dx x)}$ is locally integrable with respect to $\lambda$. Since $B_{\widetilde{U}}$ is bounded, it is sufficient to consider $B_{\widetilde{U}}^I$ for which we find
\begin{align*}
\int_{0+}^R\int_0^{z-}B_{\widetilde{U}}^I(\tfrac{z}{x})\mu(\dx x)\dx z&=\int_{0+}^R\int_{x+}^R\int_{\tfrac{z}{x}-1}^1\nu_{\widetilde{U}}(\dx y)\dx z\mu(\dx x)\\
&\leq\int_{0+}^R\int_{0+}^1\int_x^{x+xy}\dx z\nu_{\widetilde{U}}(\dx y)\mu(\dx x)\leq R\mu([0,R])\int_{0+}^1y\nu_{\widetilde{U}}(\dx y),
\end{align*}
showing that the triple integral is finite due to the finite variation condition. The other quantities can be estimated similarly. Thus, the right-hand side of~\eqref{eq-mu-fv} defines a locally finite measure, which we denote by $G^{FV}(\dx z)$. With the jump part of $\widetilde{U}$ and $\eta$ being of finite variation, it follows that $\smash{\int_{[-1,1]}|y|\nu_{\eta}(\dx y)}$ and $\smash{\int_{[-1,1]}|y|\nu_{\widetilde{U}}(\dx y)}$ are finite, in particular
\begin{align}
&\int_{\RR}\big(f(x+y)-f(x)-yf'(x)\I_{|y|\leq1}(y)\big)\nu_{\eta}(\dx y)+\gamma_{\eta}f'(x)\nonumber\\
&\quad=\int_{\RR}\big(f(x+y)-f(x)\big)\nu_{\eta}(\dx y)+\gamma_{\eta}^0f'(x)\label{integral-finitevariation}
\end{align}
with a similar relation also holding true for $\smash{\widetilde{U}}$. As in the proof of Corollary~\ref{corollary-finitemoments}, there is no need to split the integrals with respect to L\'{e}vy measures when rewriting as in Lemma~\ref{rewriting1} such that the jumps of $\eta$ and $\smash{\widetilde{U}}$, respectively, only yield a single term. Since also $\sigma_{\eta}^2=\sigma_{\widetilde{U}}^2=0$ by assumption, all terms can be rewritten to only include $f'$ and there is no need to consider antiderivatives when following the argument of Lemma~\ref{rewriting2}. This implies that the distribution~$G^{FV}$ obtained satisfies $\smash{\int_{\RR}f'(z)G^{FV}(\dx z)}=0$ for all $f\in C_c^{\infty}(\RR)$, hence $(G^{FV})'=0$, giving the form $G^{FV}(\dx z)=C\dx z$ for a single real constant~$C$. We have thus obtained an equivalent to~\eqref{eq-mu} in the finite variation case. In order for the constant to vanish, we need, similar to~\eqref{limit-constants1}, that
\begin{align}\label{limit-constants2}
C=\lim_{t\rightarrow\infty}\frac{1}{\ln(t)}\int_1^t\frac{1}{z}G^{FV}(\dx z)=0.
\end{align}
Using the results obtained in the proof of Lemma~\ref{constants}, one can directly conclude that the drift term, as well as the terms involving $B_{\eta}$ and $B_{\widetilde{U}}$ as defined in~\eqref{def-B-eta} and~\eqref{def-B-U}, respectively, satisfy the desired asymptotics, leaving only $B_{\eta}^I$ and $B_{\widetilde{U}}^I$ to be considered. Since $B_{\eta}^0\ast\mu$ is integrable with respect to $\lambda$, it readily follows that
\begin{align*}
0\leq\lim_{t\rightarrow\infty}\frac{1}{\ln(t)}\int_1^t\frac{1}{z}\big(B_{\eta}^I\ast\mu\big)(z)\dx z\leq\lim_{t\rightarrow\infty}\frac{1}{\ln(t)}\int_{\RR}\frac{1}{z}\big(B_{\eta}^I\ast\mu\big)(z)\dx z=0
\end{align*}
Further, treating $B_{\widetilde{U}}^I$ similar to $S_{\widetilde{U}}$ in Lemma~\ref{constants}, we find for $x>2z$ that
\begin{align*}
\int_{2z}^{\infty}B_{\widetilde{U}}^I(\tfrac{z}{x})\mu(\dx x)&=\int_{2z}^{\infty}\int_{-1}^{\frac{z}{x}-1}\nu_{\widetilde{U}}(\dx y)\mu(\dx x)=\int_{-1}^{-\frac{1}{2}}\int_{2z}^{\frac{z}{y+1}}\mu(\dx x)\nu_{\widetilde{U}}(\dx y)\\
&\leq\nu_{\widetilde{U}}\big([-1,-\tfrac{1}{2}]\big)\mu\big([2z,\infty]\big),
\end{align*}
which converges to zero as $z\rightarrow\infty$, and for $x\leq2z$ that
\begin{align*}
&\int_1^t\frac{1}{z}\int_{\frac{z}{2}}^{2z}B_{\widetilde{U}}^I(\tfrac{z}{x})\mu(\dx x)\dx z=\int_{\frac{1}{2}}^t\int_{\max\{\frac{x}{2},1\}}^{\min\{t,2x\}}\frac{1}{z}B_{\widetilde{U}}^I(\tfrac{z}{x})\dx z\mu(\dx x).
\end{align*}
Here, note that $B_{\widetilde{U}}^{FV}$ can be estimated by a constant if the argument is bounded away from the singularity at $z=x$ and that a suitable substitution implies
\begin{align*}
\int_{x(1-\varepsilon)}^{x(1+\varepsilon)}\frac{1}{z}B_{\widetilde{U}}^I(\tfrac{z}{x})\dx z\leq\Big(\frac{1}{1-\varepsilon}\int_{-\varepsilon}^0\nu_{\widetilde{U}}\big([-1,s)\big)\dx s+\int_0^{\varepsilon}\nu_{\widetilde{U}}\big((s,\infty)\big)\dx s\Big)\frac{1}{x}.
\end{align*}
As $\widetilde{U}$ is of finite variation by assumption, it follows that the above term is finite and thus
\begin{align*}
\lim_{t\rightarrow\infty}\frac{1}{\ln(t)}\int_1^t\frac{1}{z}\int_0^{\infty}B_{\widetilde{U}}^I(\tfrac{z}{x})\mu(\dx x)\dx z=0,
\end{align*}
yielding~\eqref{limit-constants2} as claimed.
\end{proof}

\begin{proof}[Proof of Corollary~\ref{corollary-diffdensity}]
With the jump parts of the processes $\eta$ and $\smash{\widetilde{U}}$ being of finite variation, we can follow the proof of Corollary~\ref{corollary-finitevariation} and rewrite their contribution to the distribution $G$ in terms of $B_{\eta}^{FV}$ and $B_{\widetilde{U}}^{FV}$. However, as $\sigma_{\eta}^2+\sigma_{\widetilde{U}}^2>0$, an argument similar to Lemmas~\ref{rewriting2} to~\ref{constants} is needed to find the equivalent to~\eqref{eq-mu} for the case considered. Note in particular that the desired asymptotics for~$B_{\eta}^{FV}$ and~$B_{\widetilde{U}}^{FV}$ follow directly from~\eqref{limit-constants2} in the proof of Corollary~\ref{corollary-finitevariation}. Since $\mu$ has density $f_{\mu}$ on $\RR\setminus\{0\}$ by Corollary~\ref{existence-density}, we also find an equivalent to Equation~\eqref{equation-density} which is given for a suitable constant $K\in\RR$~by
\begin{align}
\Bigl(\frac{1}{2}\sigma_{\eta}^2+\frac{1}{2}z^2\sigma_{\widetilde{U}}^2\Bigr)f_{\mu}(z)&=-K+\int_{0+}^z\bigl(\gamma_{\eta}^0+x\gamma_{\widetilde{U}}^0\bigr)f_{\mu}(x)\dx x+\I_{\{z<0\}}\gamma_{\eta}^0\mu(\{0\})\nonumber\\
&\ +\int_{0+}^z(B_{\eta}^{FV}\ast f_{\mu})(x)\dx x+\int_{0+}^zB_{\eta}^{FV}(x)\dx x\mu(\{0\})\label{eq-density-beforediff}\\
&\ +\int_{0+}^z\Big(\I_{\{t>0\}}\int_0^{\infty}B_{\widetilde{U}}^{FV}(\tfrac{t}{x})f_{\mu}(x)\dx x-\I_{\{t<0\}}\int_{-\infty}^0B_{\widetilde{U}}^{FV}(\tfrac{t}{x})f_{\mu}(x)\dx x\Big)\dx t.\nonumber
\end{align}
Here, the terms involving $\gamma_{\eta}^0$, $\gamma_{\widetilde{U}}^0$, $B_{\eta}^{FV}$ and $B_{\widetilde{U}}^{FV}$ are locally integrable with respect to the Lebesgue measure as a result of calculations similar to Lemma~\ref{lem-functions}. This implies that the respective integrals are differentiable $\lambda$-a.e. (see e.g.~\cite[Thm.~6.3.6]{Cohn2010}) such that the right-hand side of~\eqref{eq-density-beforediff} and thus~${(\frac{1}{2}\sigma_{\eta}^2+\frac{1}{2}z^2\sigma_{\widetilde{U}}^2)f_{\mu}(z)}$ is differentiable $\lambda$-a.e., implying that this must hold for $f_{\mu}$ as well. Equation~\eqref{eq-mu-diffdensity} now follows by differentiation.

Further, observe that $f_{\mu}\in C^0(\RR\setminus\{0\})$ by Corollary~\ref{existence-density}. Hence, differentiability of~$f_{\mu}$ follows by showing that the terms on the right-hand side of~\eqref{eq-density-beforediff} are in ${C^1(\RR\setminus\{0\})}$, or equivalently, by the fundamental theorem of calculus, that the functions that are integrated over $(0,z]$ are continuous on~$\RR\setminus\{0\}$. As this is trivially satisfied for the mapping ${x\mapsto\smash{(\gamma_{\eta}^0+x\gamma_{\widetilde{U}}^0)f_{\mu}(x)}}$ and the assumptions of both~(i) and~(ii) imply that~${\mu(\{0\})=0}$~(see~\cite[Cor.~6.11]{BLRRPart1}), it remains to consider the terms involving $\smash{B_{\eta}^{FV}}$ and $\smash{B_{\widetilde{U}}^{FV}}$. Similar to the treatment of $S_{\eta}$ in Corollary~\ref{existence-density}, let $x_0>0$, choose $0<\varepsilon<x_0/2$ and define $B_{\eta}^{\varepsilon}$ by replacing~$\nu_{\eta}$ by $\nu_{\eta}|_{\RR\setminus[-\varepsilon,\varepsilon]}$ in~\eqref{def-B-eta-FV}. Then~$B_{\eta}^{\varepsilon}$ is bounded and continuous in all but countably many points such that $x\mapsto \smash{(B_{\eta}^{\varepsilon}\ast f_{\mu})(x)=\int_{\RR}B_{\eta}^{\varepsilon}(x-t)f_{\mu}(t)\dx t}$ is continuous in~$x_0$ by dominated convergence. Next, observe that the remainder $\smash{B_{\eta}^{FV}-B_{\eta}^{\varepsilon}}$ is only supported on a subset of~$[-\varepsilon,\varepsilon]$ and integrable due to the finite variation condition. Therefore, the mapping
\begin{displaymath}
x\mapsto \big(\big(B_{\eta}^{FV}-B_{\eta}^{\varepsilon}\big)\ast f_{\mu}\big)(x)=\int_{-\varepsilon}^{\varepsilon}f_{\mu}(x-t)\big(B_{\eta}^{FV}-B_{\eta}^{\varepsilon}\big)(t)\dx t
\end{displaymath}
is continuous by another application of Lebesgue's dominated convergence theorem as $f_{\mu}(x-t)$ is uniformly bounded in ${x\in[x_0-\varepsilon,x_0+\varepsilon]}$ and $t\in[-\varepsilon,\varepsilon]$. Applying a similar argument for $x_0<0$, it follows that $x\mapsto(B_{\eta}^{FV}\ast f_{\mu})(x)$ is continuous on $\RR\setminus\{0\}$ as desired.

The terms involving $B_{\widetilde{U}}^{FV}$ can be treated similarly to the ones involving $S_{\widetilde{U}}$ in Corollary~\ref{existence-density}. First, decompose
\begin{equation}\label{decomp-integral-B-U-FV}
\int_0^{\infty}B_{\widetilde{U}}^{FV}(\tfrac{t}{x})f_{\mu}(x)\dx x=\int_{(0,\frac{t}{2})\cup(\frac{3}{2}t,\infty)}B_{\widetilde{U}}^{FV}(\tfrac{t}{x})f_{\mu}(x)\dx x+\int_{\frac{t}{2}}^{\frac{3}{2}t}B_{\widetilde{U}}^{FV}(\tfrac{t}{x})f_{\mu}(x)\dx x
\end{equation}
and fix $t_0>0$. Observe that $\smash{B_{\widetilde{U}}^{FV}}$ is bounded on $\RR\setminus[2/3,2]$ and that the mapping ${t\mapsto\smash{\I_{(0,\frac{t}{2})\cup(\frac{3}{2}t,\infty)}(x)B_{\widetilde{U}}^{FV}(\tfrac{t}{x})}}$ is continuous in $t_0$ for all but countably many $x>0$. Therefore, continuity of the mapping $t\mapsto\smash{\int_{\RR}\I_{(0,\frac{t}{2})\cup(\frac{3}{2}t,\infty)}(x)B_{\widetilde{U}}^{FV}(\tfrac{t}{x})f_{\mu}(x)\dx x}$ in $t_0>0$ follows by dominated convergence. Further, the second term on the right-hand side of~\eqref{decomp-integral-B-U-FV} can be rewritten using a suitable substitution, yielding
\begin{align*}
\int_{\frac{t}{2}}^{\frac{3}{2}t}B_{\widetilde{U}}^{FV}(\tfrac{t}{x})f_{\mu}(x)\dx x&=\int_{\frac{t}{2}}^t\nu_{\widetilde{U}}((\tfrac{t}{x}-1,\infty))f_{\mu}(x)\dx x-\int_t^{\frac{3}{2}t}\nu_{\widetilde{U}}((-\infty,\tfrac{t}{x}-1))f_{\mu}(x)\dx x\\
&=\int_0^1\nu_{\widetilde{U}}((w,\infty))f_{\mu}(\tfrac{t}{w+1})\frac{t}{(w+1)^2}\dx w\\
&\quad-\int_{-\frac{1}{3}}^0\nu_{\widetilde{U}}((-\infty,w))f_{\mu}(\tfrac{t}{w+1})\frac{t}{(w+1)^2}\dx w,
\end{align*}
where choosing $0<\varepsilon<t_0/2$ implies that $f_{\mu}(t/(w+1))$ can be  uniformly bounded in $t\in[t_0-\varepsilon,t_0+\varepsilon]$ and $w\in[-1/3,1]$. Since
\begin{displaymath}
\int_0^1\nu_{\widetilde{U}}((w,\infty))\frac{1}{(w+1)^2}\dx w+\int_{-\frac{1}{3}}^0\nu_{\widetilde{U}}((-\infty,w))\frac{1}{(w+1)^2}\dx w<\infty,
\end{displaymath}
the right-hand side of~\eqref{decomp-integral-B-U-FV} is continuous in $t_0>0$ by dominated convergence. Since a similar argument can be applied for $t_0<0$, the term is continuous on $\RR\setminus\{0\}$. Therefore, the right-hand side of~\eqref{eq-density-beforediff} and hence the left-hand side of~\eqref{eq-density-beforediff} are in $C^1(\RR\setminus\{0\})$, which shows that~${f_{\mu}\in C^1(\RR\setminus\{0\})}$. In particular, Equation~\eqref{eq-mu-diffdensity} holds for all $z\in\RR\setminus\{0\}$ if the assumptions of part~(i) or~(ii) of the Corollary are satisfied.
\end{proof}

\section{Applications and Examples}\label{s7}
In this section, we consider various applications of the equations in Sections~\ref{s4} and~\ref{s5}, respectively, deriving explicit information on the law of the killed exponential functional in special cases. The first example is concerned with the special case $\xi\equiv0$, which is the L\'evy process~$\eta$ subordinated by a gamma process with parameters 1 and $q>0$, evaluated at time 1. The law of $V_{q,0,\eta}$ is~$q$ times the potential measure of $\eta$, cf. \cite[Def. 30.10]{Sato2013}.
\begin{example}\label{ex-potentialmeasure}\rm{
Let $q>0$ and $\xi\equiv0$, i.e. $V_{q,0,\eta}=\eta_{\tau}$. Since $\sigma_{\xi}=\gamma_{\xi}=0$ and $\nu_{\xi}$ is the zero measure, the limit term in~\eqref{eq-limitchar} vanishes, yielding
\begin{displaymath}
\psi_\eta(u) \phi_{V_{q,\xi,\eta}}(u)=  q (\phi_{V_{q,\xi,\eta}}(u) -1)
\end{displaymath}
such that we recover the known formula for the characteristic function of the potential measure from~\cite[Prop.~37.4]{Sato2013}. One can also use the results in Section~\ref{s5} to give a distributional equation for $\mu=\cL(\eta_{\tau})$, and hence for the potential measure, by observing that the characteristics of $\smash{\widetilde{U}}$ are given by $(0,q\delta_{-1},-q)$ whenever $\xi\equiv0$. For example, if~$\eta$ is of finite variation, one obtains from~\eqref{eq-mu-fv} that
\begin{displaymath}
\gamma_{\eta}^0\mu(\dx z)+\big(B_{\eta}^{FV}\ast\mu\big)(z)\dx z+q\big(\I_{\{z<0\}}\mu((-\infty,z])-\I_{\{z>0\}}\mu([z,\infty))\big)\dx z=0,
\end{displaymath}
and if~$\sigma_{\eta}^2>0$, but the jump part of~$\eta$ is still of finite variation, it follows from Corollaries~\ref{existence-density} and~\ref{corollary-diffdensity} that~$\mu$ has a density~$f_{\mu}\in C^0(\RR)\cap C^1(\RR\setminus\{0\})$ that satisfies
\begin{displaymath}
\frac{1}{2}\sigma_{\eta}^2f_{\mu}'(z)-\gamma_{\eta}^0f_{\mu}(z)=\big(B_{\eta}^{FV}\ast f_{\mu}\big)(z)+q\Big(\I_{\{z<0\}}\int_{-\infty}^z f_{\mu}(x)\dx x-\I_{\{z>0\}}\int_z^{\infty}f_{\mu}(x)\dx x\Big).
\end{displaymath}
In the special case of $\eta$ being a standard Brownian motion, we have $\sigma_{\eta}^2=1$, $\gamma_{\eta}^0=0$ and~$B_{\eta}^{FV}$=0 and one readily checks that the solution of the differential equation is given~by
\begin{displaymath}
f_{\mu}(z)=\sqrt{\frac{q}{2}}\re^{-\sqrt{2q}|z|}=q\Big(\frac{1}{\sqrt{2q}}\re^{-\sqrt{2q}|z|}\Big),
\end{displaymath}
which is $q$ times the potential density given in~\cite[Ex.~30.11]{Sato2013}.
}\end{example}

The following example collects some cases in which the solution of Equation~\eqref{eq-diffeqchar} can be given explicitly.

\begin{example}\rm
(i) Assume that $q>0$ and that $(\xi_t)_{t\geq 0}$ is deterministic and not the zero process, i.e. $\xi_t=\gamma_{\xi} t$ with $\gamma_\xi\neq0$. We need to assure $-2\gamma_\xi<q$ and $\ew\eta_1^2<\infty$ in order to have \eqref{eq-condsecondmoment}. Under this assumption, setting~${\phi:=\phi_{V_{q,\xi,\eta}}}$, Equation \eqref{eq-diffeqchar} reduces~to
\begin{equation*}
\gamma_{\xi} u \phi'(u) + (q-\psi_\eta(u))\phi(u) =q.
\end{equation*}
For any $u>c>0$ the solution to this inhomogeneous first-order ODE is given by
\begin{align}\label{ODEsolve1}
\phi(u)&= \exp\Big(\int_c^u \frac{\psi_\eta(s)-q}{\gamma_\xi s} \dx s \Big) \Big[\phi(c) + \int_c^u \frac{q}{\gamma_\xi t} \exp\Big(- 		\int_c^t \frac{\psi_\eta(s)-q}{\gamma_\xi s} \dx s \Big)\dx t \Big].
\end{align}
Now assume that $\gamma_{\xi}>0$ and that $\psi_\eta(s)\sim\alpha s^{\beta}$ near zero for some $\alpha\in\CC\setminus\{0\}$ and $\beta>0$. Then~$\smash{\int_0^u\frac{\psi_{\eta}(s)}{\gamma_{\xi}s}\dx s}$ exists and letting $c\searrow0$ in~\eqref{ODEsolve1} leads to
\begin{equation}\label{solutionlimit-gen}
\phi(u)= \exp\Big(\int_0^u \frac{\psi_\eta(s)}{\gamma_\xi s} \dx s \Big)u^{-q/\gamma_{\xi}}\int_0^u \frac{q}{\gamma_\xi }t^{q/\gamma_{\xi}-1} \exp\Big(- \int_0^t \frac{\psi_\eta(s)}{\gamma_\xi s} \dx s \Big)\dx t
\end{equation}
for $u>0$. In the trivial case of $\psi_\eta(u)=iu$ where
\begin{equation} \label{eq-functrivial}
V_{q,\xi,\eta}=\int_0^\tau \re^{-\gamma_\xi t} \dx t = \frac{1}{\gamma_\xi} (1-\re^{-\gamma_\xi \tau}),
\end{equation}
Equation~\eqref{solutionlimit-gen} simplifies to
\begin{align*}
\phi(u)=\exp\Big(\frac{i}{\gamma_\xi} u\Big) u^{-q/\gamma_{\xi}}\int_0^u \frac{q}{\gamma_{\xi}}t^{q/\gamma_\xi-1} \exp\Big(-\frac{it}{\gamma_\xi} \Big) \dx t
\end{align*}
for $u>0$, which in the special case $q=\gamma_\xi$ can be further simplified to
\begin{displaymath}
\phi(u)=\frac{qi}{u} \Big(1-\exp\Big(\frac{u}{q} i\Big)\Big), \quad u>0,
\end{displaymath}
as characteristic function of \eqref{eq-functrivial} with $\gamma_\xi=q$. Observe from the explicit form of~$\phi$ that the characteristic function has zeroes such that the law of the killed exponential functional cannot be infinitely divisible in this case. If we assume $\eta$ to be a Brownian motion without drift instead, i.e. $\psi_\eta(u)=\smash{-\frac{\sigma_{\eta}^2}{2} u^2}$, then \eqref{eq-condsecondmoment} holds, and whenever $\gamma_{\xi}>0$, Equation~\eqref{solutionlimit-gen} reduces to
\begin{displaymath}
\phi(u)=\exp\Big(-\frac{\sigma_{\eta}^2}{\gamma_{\xi}}u^2\Big)u^{-q/\gamma_{\xi}}\int_0^u\frac{q}{\gamma_{\xi}}t^{q/\gamma_{\xi}-1}\exp\Big(\frac{\sigma_{\eta}^2}{\gamma_{\xi}}t^2\Big)\dx t.
\end{displaymath}
This can be further simplified for various values of $q/\gamma_{\xi}$, e.g. if $q=2\gamma_{\xi}$, we have
\begin{displaymath}
\phi(u)=\frac{q}{2\sigma_{\eta}^2}u^{-2}\Big(1-\exp\Big(-\frac{2\sigma_{\eta}^2u^2}{q}\Big)\Big),\quad u>0,
\end{displaymath}
and if $q=4\gamma_{\xi}$ we have
\begin{displaymath}
\phi(u)=u^{-4} \frac{q^2}{8\sigma_\eta^4} \Big[  \Big(\frac{4\sigma_\eta^2u^2}{q}-1\Big)+ \exp\Big(-\frac{4u^2\sigma_\eta^2}{q}   \Big) \Big], \quad u>0,
\end{displaymath}
as characteristic function of the killed exponential functional.

Finally, if we assume $\eta$ to be a compound Poisson process with intensity~$\lambda$ and exponentially distributed jumps with parameter $a>0$ such that $\psi_\eta(u)=\lambda \frac{iu}{a-iu}$, then we derive from \eqref{solutionlimit-gen} for $u>0$ that
\begin{align*}
\phi(u)&= (u+ia)^{-\lambda/\gamma_\xi}  u^{-q/\gamma_{\xi}}\int_0^u \frac{q}{\gamma_\xi }t^{q/\gamma_{\xi}-1} (t+ia)^{\lambda/\gamma_\xi} \dx t\\
&=  \left(\frac{a-iu}{a}\right)^{-\lambda/\gamma_\xi} {}_2F_1\Big(\frac{q}{\gamma_\xi},-\frac{\lambda}{\gamma_\xi}; 1+\frac{q}{\gamma_\xi}; \frac{iu}{a}\Big),
\end{align*}
where $_2F_1(\cdot,\cdot;\cdot;\cdot)$ denotes the hypergeometric function (see e.g.~ Formulas 15.3.1 and~15.1.1 in \cite{AbramowitzStegun1972}) and $z^{-\lambda/\gamma_\xi}$ for $z\in\CC$ is interpreted as $\smash{\exp(-\frac{\lambda}{\gamma_\xi}\log(z))}$ with $\log$ denoting the principal branch of the complex logarithm. Setting formally $q=0$ in the above expression, we obtain $\phi(u)=(\frac{a-iu}{a})^{-\lambda/\gamma_\xi}$, which is the characteristic function of the $\text{Gamma}(\lambda/\gamma_\xi,a)$-distribution. Indeed, $V_{0,\xi,\eta}$ is $\text{Gamma}(\lambda/\gamma_\xi,a)$-distributed, cf. \cite[Thm.~2.1(f)]{GjessingPaulsen1997}. Note that we can also obtain this fact from setting $q=0$ and considering $c\searrow0$ in~\eqref{ODEsolve1}.\\
(ii) Let $(\xi_t)_{t\geq 0}$ be a Brownian motion with drift~$\gamma_{\xi}$. In order to have~\eqref{eq-condsecondmoment}, we need to assume that $2(\sigma_\xi^2-\gamma_\xi)<q$ and $\ew\eta_1^2<\infty$. Under this assumption, Equation \eqref{eq-diffeqchar} leads to the following inhomogeneous second-order ODE (again setting $\phi:=\phi_{V_{q,\xi,\eta}}$)
		\begin{align*}
		\frac{\sigma_\xi^2}{2} u^2 \phi''(u) + \Big(\frac{\sigma_\xi^2}{2} -\gamma_\xi \Big) u \phi'(u) + (\psi_\eta(u)-q)\phi(u) = -q,
		\end{align*}
		which for $\gamma_\xi= \frac{\sigma_\xi^2}{2}<\frac{q}{2}$ reduces to
		\begin{align*}
		\frac{\sigma_\xi^2}{2} u^2 \phi''(u) +  (\psi_\eta(u)-q)\phi(u) = -q.
		\end{align*}				
		In particular, assuming $(\eta_t)_{t\geq 0}$ to be a  Brownian motion without drift, the resulting ODE
		\begin{align}\label{ODEsolve4} u^2 \phi''(u) -  \Big(\frac{\sigma_\eta^2}{\sigma_\xi^2}u^2+ \frac{2q}{\sigma_\xi^2}\Big)\phi(u) = - \frac{2q}{\sigma_\xi^2},\end{align}
		is a Bessel-type equation. Using the substitution $\phi_{hom}(u) = \sqrt{u}\, g_{hom}\left( \frac{\sigma_\eta}{\sigma_\xi} u\right)$ for $u>0$, it is easily checked that a function $\phi_{hom}$ satisfies the homogeneous equation corresponding to \eqref{ODEsolve4} on $(0,\infty)$ if and only if $g_{hom}$ satisfies the homogeneous modified Bessel equation $v^2 g''(v) + v g'(v) - (v^2 +  2q/\sigma_\xi^2 +1/4)g(v) = 0$ for $v\in (0,\infty)$. Denoting $\alpha := \sqrt{2q/\sigma_\xi^2 + 1/4} > 0$, two linear independent solutions of this modified Bessel equation are given by the modified Bessel functions $I_\alpha$ and $K_\alpha$ of first and second kind, respectively (cf. \cite[pp.77-78]{Watson}), hence the general solution of the  homogeneous equation corresponding to \eqref{ODEsolve4} is given by
		$$\phi_{hom}(u)=c_1\sqrt{u}I_\alpha\Big(\frac{\sigma_\eta}{\sigma_\xi}u\Big) + c_2\sqrt{u}K_\alpha\Big(\frac{\sigma_\eta}{\sigma_\xi}u\Big),\quad u > 0$$
	with complex constants $c_1,c_2$. Whenever $3\sigma_\xi^2=q$, one easily verifies that a particular solution of~\eqref{ODEsolve4} is given by
		$$\phi_{part}(u)=2\frac{q}{\sigma_\eta^2} u^{-2}=6 \frac{\sigma_\xi^2}{\sigma_\eta^2} u^{-2},$$
		and hence in this case
		\begin{align*}
		\phi(u)&= 6 \frac{\sigma_\xi^2}{\sigma_\eta^2} u^{-2}+ c_1\sqrt{u}I_{5/2}\Big(\frac{\sigma_\eta}{\sigma_\xi}u\Big) + c_2\sqrt{u}K_{5/2}\Big(\frac{\sigma_\eta}{\sigma_\xi}u\Big), \quad u > 0.
		\end{align*}
Observe that 
$$K_{5/2}(z) = \sqrt{\frac{\pi}{2z}} \re^{-z} \left( 1 + \frac{3}{z} + \frac{3}{z^2} \right)$$ and that
$I_{5/2}(z) \sim \re^z /\sqrt{2\pi z}$  as $z\to\infty$ (cf. \cite[p.80 Eqs. (10),(12)]{Watson}). Since $\phi$ is bounded as a characteristic function, we obtain $c_1=0$ when letting $u\to\infty$, and using $\lim_{u\downarrow 0} \phi(u) = 1$ we obtain $c_2 = -\sqrt{8 \sigma_\eta / (\pi \sigma_\xi)}$. Altogether we obtain
$$\phi(u) = 6 \frac{\sigma^2_\xi}{\sigma_\eta^2} u^{-2} - 2 \re^{-\sigma_\eta u/\sigma_\xi} \left( 1 + \frac{3\sigma_\xi}{\sigma_\eta u} + \frac{3\sigma_\xi^2}{\sigma_\eta^2 u^2}\right)$$
for $u>0$ whenever $\gamma_\xi = \sigma_\xi^2/2 = q/6 > 0$ and $\eta$ is a Brownian motion without drift and variance $\sigma_\eta^2$. Replacing $u$ by $|u|$ in the right-hand side the above formula also holds for $u\in \bR\setminus \{0\}$ by symmetry of $\varphi$.
\end{example}

The next example illustrates some of the results of Section~\ref{s5}.
\begin{example}\rm{
Let $q\geq0$ and assume that both $\xi$ and $\eta$ are Brownian motions with or without drift, i.e. ${\eta_t=\sigma_{\eta}B_t+\gamma_{\eta}^0t}$ and ${\xi_t=\sigma_{\xi}W_t+\gamma_{\xi}^0t}$ where $(B_t)_{t\geq0}$ and $(W_t)_{t\geq0}$ denote two independent standard Brownian motions, $\gamma_{\eta}^0,\gamma_{\xi}^0\in\RR$, $\sigma_{\eta}^2+\sigma_{\xi}^2>0$ and $\eta$ is not the zero process. It follows from Corollaries~\ref{existence-density} and~\ref{corollary-diffdensity} that the law of the killed exponential functional is absolutely continuous and that the density $f_{\mu}$ is continously differentiable on $\RR\setminus\{0\}$. Observing that the characteristics of the process $\smash{\widetilde{U}}$ are given by~$(\sigma_{\xi}^2,q\delta_{-1},\gamma_{\widetilde{U}})$ with drift~$\gamma_{\widetilde{U}}^0=-\gamma_{\xi}^0+\sigma_{\xi}^2/2$, we find from~\eqref{eq-mu-diffdensity} that $f_{\mu}$ satisfies
\begin{align}\label{ide-ex-bbdrift}
&\frac{1}{2}\big(\sigma_{\eta}^2+z^2\sigma_{\widetilde{U}}^2\big)f'_{\mu}(z)-\bigl(\gamma_{\eta}^0+z(\gamma_{\widetilde{U}}^0-\sigma_{\widetilde{U}}^2)\bigr)f_{\mu}(z)\\
&\quad+q\I_{\{z>0\}}\int_z^{\infty}f_{\mu}(x)\dx x-q\I_{\{z<0\}}\int_{-\infty}^zf_{\mu}(x)\dx x=0\nonumber
\end{align}
for $\neq0$. Note that the integral terms vanish whenever~$q=0$ such that~\eqref{ide-ex-bbdrift} reduces to an ordinary differential equation. In this case, we obtain
\begin{displaymath}
\frac{f'_{\mu}(z)}{f_{\mu}(z)}=\frac{\gamma_{\eta}^0+(\gamma_{\widetilde{U}}^0-\sigma_{\widetilde{U}}^2)z}{\frac{1}{2}\sigma_{\eta}^2+\frac{1}{2}\sigma_{\widetilde{U}}^2z^2}
\end{displaymath}
for $z\neq0$, from which the explicit solution can be derived by logarithmic integration. Assuming that~$\sigma_{\eta}^2,\sigma_{\xi}^2\neq0$, it follows that
\begin{displaymath}
f_{\mu}(z)=C\big(\sigma_{\eta}^2+z^2\sigma_{\widetilde{U}}^2\big)^{-1+\gamma_{\widetilde{U}}^0/\sigma_{\widetilde{U}}^2}\exp\Big(\frac{2\gamma_{\eta}^0}{\sigma_{\eta}\sigma_{\widetilde{U}}}\arctan\Big(\frac{\sigma_{\widetilde{U}}}{\sigma_{\eta}}x\Big)\Big)
\end{displaymath}
where $C>0$ is a norming constant. Note that even though the equation is solved for~$z>0$ and $z<0$ separately, the continuity of $f_{\mu}$ implies that the same norming constant can be used on both sides. In particular, the result obtained for $f_{\mu}$ above coincides with the density of the exponential functional given in~\cite[Thm.~2.1(d)]{GjessingPaulsen1997}. Let now $q\neq0$. In this case, the integro-differential equation~\eqref{ide-ex-bbdrift} yields an ordinary differential equation for the distribution function ${F_{\mu}(z)=\int_{-\infty}^zf_{\mu}(x)\dx x}$ of the killed exponential functional, which is given by
\begin{align*}
&\frac{1}{2}\big(\sigma_{\eta}^2+z^2\sigma_{\widetilde{U}}^2\big)F''_{\mu}(z)-\bigl(\gamma_{\eta}^0+z(\gamma_{\widetilde{U}}^0-\sigma_{\widetilde{U}}^2)\bigr)F'_{\mu}(z)-qF_{\mu}(z)=-q,\quad z>0,\\
&\frac{1}{2}\big(\sigma_{\eta}^2+z^2\sigma_{\widetilde{U}}^2\big)F''_{\mu}(z)-\bigl(\gamma_{\eta}^0+z(\gamma_{\widetilde{U}}^0-\sigma_{\widetilde{U}}^2)\bigr)F'_{\mu}(z)-qF_{\mu}(z)=0,\quad z<0.
\end{align*}
Exemplarily, we choose~$q=2$, $\sigma_{\xi}^2=4$, $\gamma_{\eta}^0=1$ and $\sigma_{\eta}^2=\gamma_{\xi}^0=0$. In this case it is $V_{q,\xi,\eta}\geq0$ a.s. due to $\eta$ being a deterministic subordinator. Hence,~\eqref{ide-ex-bbdrift} reduces to
\begin{displaymath}
2z^2f_{\mu}'(z)+(2z-1)f_{\mu}(z)+2\int_z^{\infty}f_{\mu}\dx x=0,\quad z>0
\end{displaymath}
and we find that the tail function $T_{\mu}(z)=1-F_{\mu}(z)$ satisfies
\begin{equation}\label{ode-tail}
2z^2T_{\mu}''(z)+(2z-1)T_{\mu}'(z)-2T_{\mu}(z)=0,\quad z>0.
\end{equation}
The general solution of~\eqref{ode-tail} is given by
\begin{displaymath}
T_{\mu}(z)=c_1z\re^{-1/(2z)}+c_2(2z-1)=z(c_1\re^{-1/(2z)}+2c_2)-c_2
\end{displaymath}
and it is readily checked that the constants must satisfy $c_1=2$ and $c_2=-1$ in order to obtain a tail function that satisfies $\lim_{z\downarrow 0}T_{\mu}(z)=1$ and $\lim_{z\rightarrow\infty}T_{\mu}(z)=0$. Deriving $T_{\mu}(z)=1-2z(1-\exp(-\frac{1}{2z}))$, it follows that the density is given by
\begin{displaymath}
f_{\mu}(z)=-T_{\mu}'(z)=2-\Big(\frac{1}{z}+2\Big)\exp\Big(-\frac{1}{2z}\Big)
\end{displaymath}
for $z>0$. Observe in particular that $V_{q,\xi,\eta}\overset{d}{=}Z_1/2Z_2$, where $Z_1$ is uniformly distributed on $[0,1]$ and~$\smash{Z_2\overset{d}{=}\mathrm{Exp}(1)}$ is independent of $Z_1$. This coincides with the results from~\cite[Thm.~2]{Yor92} given in Example~\ref{ex:Yor} of the introduction.
}\end{example}

Observe that, so far, both processes $\xi$ and $\eta$ were assumed to be continuous in the examples considered. We conclude this section by discussing two examples in which $\xi$, and hence $U$, is a pure-jump process.

\begin{example}\label{ex4}\rm{
Let $\xi$ be a Poisson process with intensity $c>0$ and $\eta_t=\sigma_{\eta}B_t$, where~${\sigma_{\eta}^2>0}$ and~$(B_t)_{t\geq0}$ is a standard Brownian motion. Using the connection between $\xi$ and $U$ established in Section~\ref{s2}, it is readily checked that $\sigma_{\widetilde{U}}^2=0$, $\nu_{\widetilde{U}}=c\delta_{\re^{-1}-1}+q\delta_{-1}$, as well as~$\gamma_{\widetilde{U}}^0=-\gamma_{\xi}^0=0$, and it follows that
\begin{align*}
B_{\widetilde{U}}^{FV}=\begin{cases}-q,\quad&\text{if}\ z\in(0,\tfrac{1}{\re}],\\ -(c+q),\quad&\text{if}\ z\in(\tfrac{1}{e},1),\\ 0,\quad&\text{if}\ z\geq1\ \text{or}\ z=0,\end{cases}
\end{align*}
for the function $B_{\widetilde{U}}^{FV}$ as defined in Corollary~\ref{corollary-finitevariation}. By Corollaries~\ref{existence-density} and~\ref{corollary-diffdensity}, $\mu$ has a density $f_{\mu}\in C^0(\RR)\cap C^1(\RR\setminus\{0\})$ that satisfies
\begin{align}
\frac{\sigma_{\eta}^2}{2}f_{\mu}'(z)&=-\I_{\{z>0\}}\Big(q\int_z^{\infty}f_{\mu}(x)\dx x+c\int_z^{\re z}f_{\mu}(x)\dx x\Big)\nonumber\\
&\quad+\I_{\{z<0\}}\Big(q\int_{-\infty}^zf_{\mu}(x)\dx x+c\int_{\re z}^zf_{\mu}(x)\dx x\Big)\label{eq-ex4}
\end{align}
for $z\neq0$. Observe that in particular $\mu(\{0\})=0$ as a consequence of $\sigma_{\eta}^2>0$. Since the right-hand side of~\eqref{eq-ex4} is differentiable, so is the left-hand side, such that we obtain~${f_{\mu}\in C^2(\RR\setminus\{0\})}$, as well as
\begin{displaymath}
\frac{\sigma_{\eta}^2}{2}f_{\mu}''(z)=qf_{\mu}(z)+c\big(f_{\mu}(z)-f_{\mu}(\re z)\big)
\end{displaymath}
for $z\neq0$ by differentiating~\eqref{eq-ex4}.
}\end{example}

\begin{example}\rm{
Assume now that $\xi$ is a compound Poisson process with L\'{e}vy measure~$\nu_{\xi}(\dx x)=\re^{-x}\I_{(0,\infty)}\dx x$, $\eta_t=\sigma_{\eta}B_t+\gamma_{\eta}t$, where $\sigma_{\eta}^2>0$ and $(B_t)_{t\geq0}$ again denotes a standard Brownian motion, as well as $q>0$. As in Example~\ref{ex4}, it follows from Corollaries~\ref{existence-density} and~\ref{corollary-diffdensity} that $\mu$ is absolutely continuous with density $f_{\mu}\in C^0(\RR)\cap C^1(\RR\setminus\{0\})$. Using the relation between $\nu_{\xi}$ and $\nu_{\widetilde{U}}$, we can give the function~$B_{\widetilde{U}}^{FV}$ as
\begin{align*}
B_{\widetilde{U}}^{FV}(z)=\begin{cases}0,\quad &z>1,\\ -(z+q),\quad& z\in[0,1),\end{cases}
\end{align*}
such that Equation~\eqref{eq-mu-diffdensity} reads
\begin{align*}
\frac{1}{2}\sigma_{\eta}^2f_{\mu}'(z)=\gamma_{\eta}f_{\mu}(z)-\I_{\{ z>0\}}\int_z^{\infty}\Big(\frac{z}{x}+q\Big)f_{\mu}(x)\dx x+\I_{\{ z<0\}}\int_{-\infty}^z\Big(\frac{z}{x}+q\Big)f_{\mu}(x)\dx x.
\end{align*}
Since $f_{\mu}\in C^0(\RR)\cap C^1(\RR\setminus\{0\})$, the integral terms are differentiable in $z\neq0$ and it follows that $f_{\mu}'\in C^1(\RR\setminus\{0\})$. Thus, $f_{\mu}\in C^2(\RR\setminus\{0\})$ and differentiating the equation leads to
\begin{align*}
\frac{1}{2}\sigma_{\eta}^2f_{\mu}''(z)&=\gamma_{\eta}f_{\mu}'(z)-\I_{\{z>0\}}\Big(\int_z^{\infty}\frac{1}{x}f_{\mu}(x)\dx x-(1+q)f_{\mu}(z)\Big)\\
&\quad +\I_{\{z<0\}}\Big(\int_{-\infty}^z\frac{1}{x}f_{\mu}(x)\dx x+(1+q)f_{\mu}(z)\Big)
\end{align*}
for $z\neq0$. This shows that $f_{\mu}''\in C^1(\RR\setminus\{0\})$ and hence $f_{\mu}\in C^3(\RR\setminus\{0\})$. Differentiating the equation once more finally eliminates the integrals and leads to the third-order linear~ODE
\begin{align*}
\frac{1}{2}\sigma_{\eta}^2f_{\mu}'''(z)=\gamma_{\eta}f_{\mu}''(z)+(1+q)f_{\mu}'(z)+\frac{1}{z}f_{\mu}(z),
\end{align*}
which is satisfied for all $z\neq0$.
}\end{example}

%%%%%%%%%%%%%%%%%%%%%%%%%%%%%%%%%%%%%%%%%%%%%

\end{document}